\documentclass{article}

\usepackage[nonatbib,preprint]{neurips_2022}

\usepackage[utf8]{inputenc} 
\usepackage[T1]{fontenc}    
\usepackage{hyperref}       

\hypersetup{
	colorlinks,
	citecolor=blue,
	linkcolor=red,
	urlcolor=blue}

\usepackage{url}            
\usepackage{booktabs}       
\usepackage{amsfonts}       
\usepackage{nicefrac}       
\usepackage{microtype}      
\usepackage{xcolor}         

\usepackage{times}
\usepackage{epsfig}
\usepackage{graphicx}
\usepackage{tikz}
\usepackage{subfig}

\usepackage[linesnumbered,lined,boxed,commentsnumbered,ruled,vlined]{algorithm2e}

\SetCommentSty{mycommfont}

\usepackage{booktabs}

\usepackage{amsfonts}
\usepackage{bm}
\usepackage{amsmath,amsthm,amssymb,rotating, mathrsfs}

\usepackage{footmisc}

\usepackage{mathtools}
\usepackage{cancel} 

\usepackage{xfrac} 
\usepackage{xparse} 

\DeclareMathAlphabet{\mathpzc}{OT1}{pzc}{m}{it}
\usepackage{enumitem}

\newcommand{\cK}{\mathcal{K}}

\newcommand{\cM}{\mathcal{M}}
\newcommand{\cN}{\mathcal{N}}

\newcommand{\cT}{\mathcal{T}}

\newcommand{\cV}{\mathcal{V}}

\newcommand{\bbE}{\mathbb{E}}

\newcommand{\bbR}{\mathbb{R}}
\newcommand{\bbS}{\mathbb{S}}

\newcommand{\bPi}{\bm{\Pi}}
\newcommand{\bA}{\bm{A}}

\newcommand{\bD}{\bm{D}}

\newcommand{\bF}{\bm{F}}

\newcommand{\bI}{\bm{I}}

\newcommand{\ba}{\bm{a}}

\newcommand{\bd}{\bm{d}}

\newcommand{\bg}{\bm{g}}

\newcommand{\br}{\bm{r}}

\newcommand{\bv}{\bm{v}}
\newcommand{\bw}{\bm{w}}
\newcommand{\bx}{\bm{x}}
\newcommand{\by}{\bm{y}}
\newcommand{\bz}{\bm{z}}

\NewDocumentCommand{\norm}{mG{2}}{\big\|#1\big\|_{#2}}

\newcommand{\trans}{\top}
\newcommand{\trsp}[1]{#1^\trans}

\newcommand{\argmin}{\mathop{\rm argmin}}

\NewDocumentCommand{\seqp}{mG{n}}{{#1}_1-\cdots+ {#1}_{#2}}
\NewDocumentCommand{\seqm}{mG{n}}{{#1}_1-\cdots- {#1}_{#2}}

\newcommand{\myparagraph}[1]{\smallskip\noindent\textbf{#1.}}

\newtheorem{theorem}{Theorem}

\newtheorem{prop}{Proposition}

\newtheorem{lemma}{Lemma}

\theoremstyle{definition}
\newtheorem{definition}{Definition}
\newtheorem{example}{Example}

\theoremstyle{remark}

\title{Global Linear and Local Superlinear Convergence of IRLS for Non-Smooth Robust Regression}

\author{%
  Liangzu Peng \\
  Mathematical Institute for Data Science\\
  Johns Hopkins University\\
  \texttt{lpeng25@jhu.edu} \\
  \And
  Christian K\"ummerle\\
  Department of Computer Science \\
  University of North Carolina at Charlotte \\
  \texttt{kuemmerle@uncc.edu} \\
  \And
  Ren\'e Vidal \\
  Mathematical Institute for Data Science \\
  Johns Hopkins University\\
  \texttt{rvidal@jhu.edu}
}

\begin{document}

\maketitle


\begin{abstract}
  We advance both the theory and practice of robust $\ell_p$-quasinorm regression 
  for $p \in (0,1]$ by using novel variants of \textit{iteratively reweighted least-squares} (IRLS) to solve the underlying non-smooth problem. In the convex case, $p=1$, we prove that this IRLS variant converges globally at a linear rate under a mild, deterministic condition on the feature matrix called the \textit{stable range space property}. In the non-convex case, $p\in(0,1)$, we prove that under a similar condition, IRLS converges locally to the global minimizer at a superlinear rate of order $2-p$; the rate becomes quadratic as $p\to 0$. We showcase the proposed methods in three applications: real phase retrieval, regression without correspondences, and robust face restoration. The results show that (1) IRLS can handle a larger number of outliers than other methods, (2) it is faster than competing methods at the same level of accuracy, (3) it restores a sparsely corrupted face image with satisfactory visual quality. \url{https://github.com/liangzu/IRLS-NeurIPS2022}
\end{abstract}

\section{Introduction}
Given a feature matrix $\bA\in\bbR^{m\times n}$ with $m\gg n$ and a response vector $\by\in\bbR^m$, the problem
\begin{align}\label{eq:Lp}
	\min_{\bx \in\bbR^{n}} \norm{\bA \bx - \by}{p}
\end{align}
of $\ell_p$-regression has a variety of different applications depending on the choice of $\|\bv\|_p = \big(\sum_{i} |\bv_i|^p \big)^{1/p}$. While $p=2$ corresponds to standard linear regression, the choice of $p>2$ arises naturally in semi-supervised learning on graphs \cite{ElAlaoui-ABLpLaplacian2016,Adil-NeurIPS2019,Slepcev-SIAMAnalysis19}, and a lot of activity has been dedicated recently to the computational complexity analysis for the case $p>1$ \cite{Bubeck-STOC2018,Adil-SODA2019,Jambulapati-arXiv2021}.

In this paper, we assume there is a coefficient vector $\bx^*\in\bbR^n$ such that the residual $\br^*:=\bA \bx^* - \by$ is $k$-sparse, in which case a choice of $p \in (0,1]$ is of interest. Indeed, for the convex and non-smooth case $p=1$, \eqref{eq:Lp} is known as \textit{least absolute deviation}, which dates back to the time of Boscovich around in the middle of the 18th century \cite{Boscovich-1757,Plackett-Biometrika1972}. Since then, it has been known intuitively that \eqref{eq:Lp} ($p=1$) is robust to \textit{large but few measurement errors} (quoting \cite{Wright-2020}), i.e., that \eqref{eq:Lp} is robust to outliers. 

The algorithmic and theoretical understanding of \eqref{eq:Lp} for $p=1$ has been of long-standing interest to statisticians, which has led to a vast literature spanning from the classic 1964 paper of Huber \cite{Huber-1964} to recent contributions \cite{Candes-FOCS2005,Bhatia-NeurIPS2015,Bhatia-NeurIPS2017,Yang-JMLR2017,Durfee-COLT2018,Suggala-COLT2019,Mukhoty-AISTATS2019,Pesme-NeurIPS2020,Chen-COLT2021,Parulekar-arXiv2021}. That being said, for $p=1$, there is a close, but somewhat underexplored connection between \eqref{eq:Lp} and the \textit{basis pursuit} problem  \cite{Candes-TIT2005}, and more specifically, the theoretical \cite{Candes-CPAM2006,Donoho-CPAM2006,Cohen-JAMS2009} and algorithmic aspects \cite{Blumensath-ACHA2009,Daubechies-CPAM2010,Kummerle-NeurIPS2021} of compressed sensing \cite{Foucart-2013} (see also Section \ref{subsection:connection-CS}). For $p\in(0,1)$, problem \eqref{eq:Lp} is non-smooth and non-convex, and generally less well understood. 
To the best of our knowledge, the only paper that considered \eqref{eq:Lp} with $p\in(0,1)$ is \cite{Chartrand-ICASSP2007}, where the authors presented a condition (based on \textit{restricted isometry constants}) that guarantees exact recovery of $\bx^*$ from \eqref{eq:Lp}. Other related works come either from the compressed sensing literature, where an $\ell_p$ (and noisy) version of basis pursuit has been considered as a sparsity-promoting formulation \cite{Chartrand-SPL2007,Chartrand-2008,Foucart-ACHA2009,Daubechies-CPAM2010,Wang-TIT2011,Sun-ACHA2012,Ba-TSP2013,Zheng-TIT2017}, or from the literature of matrix recovery/completion, where the \textit{Schatten-$p$} norm has come into play as a non-convex surrogate for rank minimization \cite{Marjanovic-TSP2012,Mohan-JMLR2012,Nie-AAAI2012,Kummerle-JMLR2018,Giampouras-NeurIPS2020,Kummerle-ICML2021}. A key message from these works is that $\ell_p$ or Schatten-$p$ minimization with $p\in(0,1)$ offers better information-theoretic properties (e.g., requires fewer samples for exact recovery) than minimization with $p=1$.
\begin{algorithm}[t]
	\SetAlgoLined
	\DontPrintSemicolon
	Input: $\bA=\trsp{[\ba_1, \dots,\ba_m]}\in\bbR^{m\times n}, \by=\trsp{[y_1,\dots,y_m]}\in\bbR^{m}$, $p\in(0,1]$;
	
	Weight initialization $\bw^{(0)}\gets[w^{(0)}_1;\dots; w^{(0)}_m]\in\bbR^m$; \tcp*{Instead, one can initialize some vector $\bx^{(1)}$.} 
	For $t\gets 0,1,\dots$:
	\vspace{-0.3cm}
	\begin{align}
			\bx^{(t+1)} &\gets \argmin_{\bx\in\bbR^n} \sum\nolimits_{i=1}^{m} w^{(t)}_{i} (\trsp{\ba_i}\bx - y_i)^2 \label{eq:x-update} \\
		\textnormal{Update}& \textnormal{ \textcolor{blue}{$\epsilon^{(t+1)}$} suitably based on $\bx^{(t+1)}$, $\bA$, and $\by$;}  \textnormal{\tcp*{See Section \ref{subsection:IRLS-basic} for details. }}   \nonumber \\
			w^{(t+1)}_i &\gets \max\big\{ |\trsp{\ba_i}\bx^{(t+1)} - y_i|, \textnormal{\textcolor{blue}{$\epsilon^{(t+1)}$}} \big\}^{p-2} \ \ \ \forall i=1,\dots,m \label{eq:w-update}
	\end{align}
	\vspace*{-0.45cm}
	\caption{IRLS for $\ell_p$-Regression ($\texttt{IRLS}_p$) } \label{algo:IRLS}
\end{algorithm}

Here, we study an \textit{iteratively reweighted least-squares} method ($\texttt{IRLS}_p$) to solve \eqref{eq:Lp} with $p\in(0,1]$. As listed in Algorithm \ref{algo:IRLS}, $\texttt{IRLS}_p$ alternates between solving a weighted least-squares problem \eqref{eq:x-update} and updating the weights \eqref{eq:w-update}; see Section \ref{subsection:IRLS-basic} for more elaboration on IRLS. The simplicity of this idea (with its impressive performance) justifies its popularity in many machine learning \cite{Lerman-FoCM2015,Tsakiris-JMLR2018,Zhu-NeurIPS2018,Ding-ICML2019} and computer vision \cite{Aftab-WCACV2015,Dong-PRL2019,Iwata-ECCV2020,Ding-CVPR2020,Yang-RA-L2020} applications. Based on recent advances on IRLS \cite{Daubechies-CPAM2010,Kummerle-NeurIPS2021}, we make several contributions for understanding \eqref{eq:Lp} and $\texttt{IRLS}_p$. We state our contributions next.

\myparagraph{The Stable Range Space Property (Section \ref{subsection:RSP})} We put forward the use of the (stable) \textit{range space property} (RSP) for studying the robust regression problem \eqref{eq:Lp}. The stable RSP was proposed by \cite{Zhao-SIAM-J-O2012} in a different context to analyze a compressed sensing algorithm for solving a weighted basis pursuit problem at each iteration. 
In analogy to the \textit{nullspace property} in compressed sensing \cite{Cohen-JAMS2009,Foucart-2013}, we show that the RSP is a necessary and sufficient condition for guaranteeing that the $\ell_p$-regression problem in \eqref{eq:Lp} admits $\bx^*$ as its unique solution (Proposition \ref{prop:RSP=exact-recovery}). Moreover, we show that if $\bA\in\bbR^{m\times n}$ has i.i.d. $\cN(0,1)$ entries and if $m$ is large enough, then the (stable) RSP holds with high probability (Proposition \ref{prop:Gaussian-RSP}). This justifies 
its use as the core assumption in our analysis.

\myparagraph{Global Linear Convergence (Section \ref{subsection:GLC-noiseless})} We prove in Theorem \ref{theorem:GLC-L1} that, under a stable RSP assumption and with $\epsilon^{(t)}$ suitably updated, the $\texttt{IRLS}_p$ Algorithm \ref{algo:IRLS} with $p=1$ converges linearly to the ground-truth $\bx^*$ from any initial weight $\bw^{(0)}$ (or equivalently from any initial point $\bx^{(1)}$). Note that while (accelerated) first-order methods (e.g., (sub-)gradient descent and proximal algorithms) can also solve the convex and non-smooth problem \eqref{eq:Lp} with $p=1$, they exhibit, in general, (global) sub-linear rates at best \cite{Beck-OptBook2017,Nesterov-2018}. To our knowledge, \cite{Mukhoty-AISTATS2019} is the only paper that claims global linear convergence of a different IRLS variant for \eqref{eq:Lp} with $p=1$; we compare our results with the ones of \cite{Mukhoty-AISTATS2019} in Section \ref{subsection:GLC-noiseless}. On the other hand, Theorem \ref{theorem:GLC-L1} is inspired by the IRLS method of \cite{Kummerle-NeurIPS2021} for basis pursuit; based on \cite{Kummerle-NeurIPS2021}, we suitably modify their proof strategy, and thus obtain a faster linear rate.


\myparagraph{Local Superlinear Convergence (Section \ref{subsection:LSC-noiseless})} We prove in Theorem \ref{theorem:LSC-noiseless} that, under a stable RSP assumption and with $\epsilon^{(t)}$ suitably updated, the $\texttt{IRLS}_p$ Algorithm \ref{algo:IRLS} with $p\in(0,1)$ converges to $\bx^*$ superlinearly, provided that $\bx^{(1)}$ falls into a certain neighborhood of $\bx^*$. To the best of our knowledge, no similar result exists for $\ell_p$-regression \eqref{eq:Lp}. While Theorem \ref{theorem:LSC-noiseless} is inspired by the IRLS algorithm of \cite{Daubechies-CPAM2010} for $\ell_p$-basis pursuit, their IRLS method does not work well for small $p$ (Section \ref{subsection:IRLS-basic}), and unlike in our work, their radius of local convergence diminishes greatly as $p \to 0$ or $m\to \infty$.

\myparagraph{Applications (Section \ref{section:applications})} We illustrate the performance of $\texttt{IRLS}_p$ for \textit{real phase retrieval} \cite{Chen-MP2019,Tan-AA-J-IMA2019}, \textit{linear regression without correspondences} \cite{Unnikrishnan-TIT18,Slawski-JoS19,Tsakiris-TIT2020,Peng-SPL2020}, and \textit{face restoration from sparsely corrupted measurements} \cite{Wright-TPAMI2008,Solomon-2011,Yao-NeurIPS2021}.  For real phase retrieval (Section \ref{subsection:RPR}), we show that $\texttt{IRLS}_{0.1}$ needs only $m=2n-1$ measurements to recover $\bx^*$ up to sign, with an additional assumption. Theoretically, even brute-force can fail to identify $\pm\bx^*$ for fewer than $2n-1$ measurements \cite{Balan-ACHA2006}. Empirically, many methods, including \textit{Kaczmarz} \cite{Strohmer-JFAA2009,Wei-IP2015,Tan-AA-J-IMA2019}, \textit{PhaseLamp} \cite{Dhifallah-Allerton2017}, \textit{truncated Wirtinger flow} \cite{Chen-NIPS2015}, and \textit{coordinate descent} \cite{Zeng-arXiv2017}, fail with $m=2n-1$ Gaussian measurements (Figure \ref{fig:RPR}). For linear regression without correspondences (Section \ref{subsection:SLR}), we show in Figures \ref{fig:SLR1}-\ref{fig:SLR2} that, $\texttt{IRLS}_{0.1}$ is uniformly faster (20-100x) and more accurate than \textit{PDLP} \cite{Applegate-NeurIPS2021} (merged into \hyperlink{https://github.com/google/or-tools}{Google or-tools}) and the commercial solver \textit{Gurobi} \cite{Gurobi951}, both of which solve \eqref{eq:Lp} with $p=1$ as a linear program, and also than \textit{subgradient descent} implemented by Beck \& Guttmann-Beck \cite{Beck-OMS2019}. For face restoration (Section \ref{subsection:face}), we present both quantitative and qualitative results on the Extended Yale B dataset \cite{Georghiades-PAMI2001}.

\section{Background}\label{section:background}
\subsection{Connection of Robust Regression and Compressed Sensing}\label{subsection:connection-CS}
The $\ell_p$-regression problem \eqref{eq:Lp} has a natural correspondence to the sparse recovery problem
\begin{align}\label{eq:Lp-BP}
	\min_{\br\in\bbR^{m}} \norm{\br}{p} \ \ \ \ \ \ \text{s.t.} \ \ \ \ \ \  \bD \br = \bz
\end{align}
where $\bD$ is a $(m-n) \times m$ matrix, $\bz\in\bbR^{m-n}$ a vector and the objective $\| \br \|_{p}$ penalizes coefficient vectors with too many non-zero coordinates. For $p=1$, \eqref{eq:Lp-BP} is also called \emph{basis pursuit} \cite{Chen-SIAM-Rev2001}. Problems \eqref{eq:Lp} and \eqref{eq:Lp-BP} are known to be related in the following sense (as implicitly stated in \cite{Candes-TIT2005,Chartrand-ICASSP2007}):
\begin{prop}\label{prop:LpRR=LpBP}
	Suppose that the range space of $\bA$ is equal to the nullspace of $\bD$, and that $\bz = -\bD \by$. If $\bx_1$ globally minimizes \eqref{eq:Lp}, then $\br_1:=\bA \bx_1 - \by$ globally minimizes \eqref{eq:Lp-BP}. On the other hand, if $\br_2$ globally minimizes \eqref{eq:Lp-BP}, then there exists some $\bx_2$ with $\br_2=\bA \bx_2-\by$ that globally minimizes \eqref{eq:Lp}.
\end{prop}
Proposition \ref{prop:LpRR=LpBP} sheds light on how we ``transfer'', with new insights, from the analysis of \cite{Daubechies-CPAM2010,Kummerle-NeurIPS2021} for \eqref{eq:Lp-BP} to results for \eqref{eq:Lp}. In what follows, we treat \cite{Daubechies-CPAM2010,Kummerle-NeurIPS2021} in the context of robust regression 
and highlight the contribution of our work relative to \cite{Daubechies-CPAM2010,Kummerle-NeurIPS2021} and other existing results, whenever possible.%

\subsection{Iteratively Reweighted Least-Squares: The Basics and New Insights}\label{subsection:IRLS-basic}
We first discuss two different updating rules for the \textit{smoothing parameter} $\epsilon^{(t)}$ of Algorithm \ref{algo:IRLS}: the \textit{fixed} rule and the \textit{dynamic} rule. We then consider the weight updating strategy \eqref{eq:w-update}. Along the way we use synthetic experiments to illustrate ideas; see Appendix \ref{section:experimental-setup} for the experimental setup. In doing so, we intend to provide a review of the state-of-the-art on the variants of IRLS.

\myparagraph{Fixed Smoothing Parameter} Instead of updating $\epsilon^{(t)}$ at each iteration, most works on IRLS use a fixed and small positive number, e.g., $\epsilon:=\epsilon^{(t)}=0.001$ for each $t$ \cite{Lerman-FoCM2015,Lerman-J-IMA2018,Tsakiris-JMLR2018,Ding-ICML2019,Dong-PRL2019,Ding-CVPR2020,Iwata-ECCV2020,Qu-arXiv2020}. The intuition is to avoid division by the potentially very small residual $|\trsp{\ba_i}\bx^{(t)} - y_i|$. But
this common practice of fixing $\epsilon$ comes with at least three issues.
First, IRLS with a fixed $\epsilon>0$ converges only to an ``$\epsilon$-approximate'' point, not exactly the global minimizer (\cite{Chan-SIAM-J-NA1999,Beck-SIAMOpt2015,Lerman-FoCM2015,Lerman-J-IMA2018}, Figure \ref{fig:L1RR}). Second, even if setting $\epsilon$ small (e.g., $\epsilon=10^{-15}$) could lead to an accurate enough solution for $p=1$ (Figure \ref{fig:L1RR}), it can fail for $p<1$ (Figure \ref{fig:LpRR}). Finally, it makes obtaining a global linear convergence rate guarantee difficult: We are not aware of any theory about a global linear rate of IRLS with fixed $\epsilon$.

\begin{figure}
	\centering
	\subfloat[$p=1$]{\input{./figures/experimentL1RR.tex} \label{fig:L1RR} } 
	\subfloat[$p<1$]{\input{./figures/experimentLpRR.tex} \label{fig:LpRR} }
	\subfloat[Sensitivity to $\alpha$ ($k=200$)]{
\begin{tikzpicture}[x=1pt,y=1pt]
\definecolor{fillColor}{RGB}{255,255,255}
\path[use as bounding box,fill=fillColor,fill opacity=0.00] (0,0) rectangle (126.47,112.02);
\begin{scope}
\path[clip] (  0.00,  0.00) rectangle (126.47,112.02);
\definecolor{drawColor}{RGB}{255,255,255}
\definecolor{fillColor}{RGB}{255,255,255}

\path[draw=drawColor,line width= 0.6pt,line join=round,line cap=round,fill=fillColor] ( -0.00,  0.00) rectangle (126.47,112.02);
\end{scope}
\begin{scope}
\path[clip] ( 35.24, 31.17) rectangle (120.97,106.52);
\definecolor{fillColor}{gray}{0.92}

\path[fill=fillColor] ( 35.24, 31.17) rectangle (120.97,106.52);
\definecolor{drawColor}{RGB}{0,0,0}

\path[draw=drawColor,line width= 0.6pt,dash pattern=on 4pt off 4pt ,line join=round] ( 39.90,103.09) --
	( 49.07, 35.22) --
	( 58.24, 35.25) --
	( 67.41, 35.10) --
	( 76.58, 34.97) --
	( 85.75, 35.57) --
	( 94.92, 35.31) --
	(104.09, 35.26) --
	(113.25, 35.30);

\path[draw=drawColor,line width= 0.6pt,dash pattern=on 4pt off 4pt ,line join=round] ( 41.43, 99.35) --
	( 50.60, 35.02) --
	( 59.77, 34.78) --
	( 68.94, 34.60) --
	( 78.10, 35.13) --
	( 87.27, 35.17) --
	( 96.44, 35.31) --
	(105.61, 35.31) --
	(114.78, 35.13);

\path[draw=drawColor,line width= 0.6pt,dash pattern=on 4pt off 4pt ,line join=round] ( 42.95, 97.12) --
	( 52.12, 34.78) --
	( 61.29, 35.03) --
	( 70.46, 34.82) --
	( 79.63, 35.27) --
	( 88.80, 35.19) --
	( 97.97, 35.46) --
	(107.14, 35.31) --
	(116.31, 35.19);
\definecolor{fillColor}{RGB}{255,0,0}

\path[draw=drawColor,line width= 0.4pt,line join=round,line cap=round,fill=fillColor] ( 38.35,101.54) rectangle ( 41.45,104.64);

\path[draw=drawColor,line width= 0.4pt,line join=round,line cap=round,fill=fillColor] ( 47.52, 33.67) rectangle ( 50.62, 36.77);

\path[draw=drawColor,line width= 0.4pt,line join=round,line cap=round,fill=fillColor] ( 56.69, 33.70) rectangle ( 59.79, 36.80);

\path[draw=drawColor,line width= 0.4pt,line join=round,line cap=round,fill=fillColor] ( 65.86, 33.55) rectangle ( 68.96, 36.65);

\path[draw=drawColor,line width= 0.4pt,line join=round,line cap=round,fill=fillColor] ( 75.03, 33.42) rectangle ( 78.13, 36.52);

\path[draw=drawColor,line width= 0.4pt,line join=round,line cap=round,fill=fillColor] ( 84.20, 34.02) rectangle ( 87.30, 37.12);

\path[draw=drawColor,line width= 0.4pt,line join=round,line cap=round,fill=fillColor] ( 93.37, 33.76) rectangle ( 96.46, 36.86);

\path[draw=drawColor,line width= 0.4pt,line join=round,line cap=round,fill=fillColor] (102.54, 33.72) rectangle (105.63, 36.81);

\path[draw=drawColor,line width= 0.4pt,line join=round,line cap=round,fill=fillColor] (111.71, 33.75) rectangle (114.80, 36.85);

\path[draw=drawColor,line width= 0.4pt,line join=round,line cap=round,fill=fillColor] ( 41.43, 99.35) circle (  1.75);

\path[draw=drawColor,line width= 0.4pt,line join=round,line cap=round,fill=fillColor] ( 50.60, 35.02) circle (  1.75);

\path[draw=drawColor,line width= 0.4pt,line join=round,line cap=round,fill=fillColor] ( 59.77, 34.78) circle (  1.75);

\path[draw=drawColor,line width= 0.4pt,line join=round,line cap=round,fill=fillColor] ( 68.94, 34.60) circle (  1.75);

\path[draw=drawColor,line width= 0.4pt,line join=round,line cap=round,fill=fillColor] ( 78.10, 35.13) circle (  1.75);

\path[draw=drawColor,line width= 0.4pt,line join=round,line cap=round,fill=fillColor] ( 87.27, 35.17) circle (  1.75);

\path[draw=drawColor,line width= 0.4pt,line join=round,line cap=round,fill=fillColor] ( 96.44, 35.31) circle (  1.75);

\path[draw=drawColor,line width= 0.4pt,line join=round,line cap=round,fill=fillColor] (105.61, 35.31) circle (  1.75);

\path[draw=drawColor,line width= 0.4pt,line join=round,line cap=round,fill=fillColor] (114.78, 35.13) circle (  1.75);

\path[draw=drawColor,line width= 0.4pt,line join=round,line cap=round,fill=fillColor] ( 42.95, 94.93) --
	( 45.15, 97.12) --
	( 42.95, 99.31) --
	( 40.76, 97.12) --
	cycle;

\path[draw=drawColor,line width= 0.4pt,line join=round,line cap=round,fill=fillColor] ( 52.12, 32.59) --
	( 54.31, 34.78) --
	( 52.12, 36.97) --
	( 49.93, 34.78) --
	cycle;

\path[draw=drawColor,line width= 0.4pt,line join=round,line cap=round,fill=fillColor] ( 61.29, 32.84) --
	( 63.48, 35.03) --
	( 61.29, 37.22) --
	( 59.10, 35.03) --
	cycle;

\path[draw=drawColor,line width= 0.4pt,line join=round,line cap=round,fill=fillColor] ( 70.46, 32.63) --
	( 72.65, 34.82) --
	( 70.46, 37.01) --
	( 68.27, 34.82) --
	cycle;

\path[draw=drawColor,line width= 0.4pt,line join=round,line cap=round,fill=fillColor] ( 79.63, 33.08) --
	( 81.82, 35.27) --
	( 79.63, 37.46) --
	( 77.44, 35.27) --
	cycle;

\path[draw=drawColor,line width= 0.4pt,line join=round,line cap=round,fill=fillColor] ( 88.80, 33.00) --
	( 90.99, 35.19) --
	( 88.80, 37.38) --
	( 86.61, 35.19) --
	cycle;

\path[draw=drawColor,line width= 0.4pt,line join=round,line cap=round,fill=fillColor] ( 97.97, 33.27) --
	(100.16, 35.46) --
	( 97.97, 37.65) --
	( 95.78, 35.46) --
	cycle;

\path[draw=drawColor,line width= 0.4pt,line join=round,line cap=round,fill=fillColor] (107.14, 33.12) --
	(109.33, 35.31) --
	(107.14, 37.50) --
	(104.95, 35.31) --
	cycle;

\path[draw=drawColor,line width= 0.4pt,line join=round,line cap=round,fill=fillColor] (116.31, 33.00) --
	(118.50, 35.19) --
	(116.31, 37.38) --
	(114.12, 35.19) --
	cycle;
\end{scope}
\begin{scope}
\path[clip] (  0.00,  0.00) rectangle (126.47,112.02);
\definecolor{drawColor}{gray}{0.10}

\node[text=drawColor,anchor=base east,inner sep=0pt, outer sep=0pt, scale=  0.73] at ( 30.29,100.57) {$10^{-2}$};

\node[text=drawColor,anchor=base east,inner sep=0pt, outer sep=0pt, scale=  0.73] at ( 30.29, 90.45) {$10^{-4}$};

\node[text=drawColor,anchor=base east,inner sep=0pt, outer sep=0pt, scale=  0.73] at ( 30.29, 29.72) {$10^{-16}$};
\end{scope}
\begin{scope}
\path[clip] (  0.00,  0.00) rectangle (126.47,112.02);
\definecolor{drawColor}{gray}{0.20}

\path[draw=drawColor,line width= 0.6pt,line join=round] ( 32.49,103.60) --
	( 35.24,103.60);

\path[draw=drawColor,line width= 0.6pt,line join=round] ( 32.49, 93.48) --
	( 35.24, 93.48);

\path[draw=drawColor,line width= 0.6pt,line join=round] ( 32.49, 32.75) --
	( 35.24, 32.75);
\end{scope}
\begin{scope}
\path[clip] (  0.00,  0.00) rectangle (126.47,112.02);
\definecolor{drawColor}{gray}{0.20}

\path[draw=drawColor,line width= 0.6pt,line join=round] ( 41.43, 28.42) --
	( 41.43, 31.17);

\path[draw=drawColor,line width= 0.6pt,line join=round] ( 59.77, 28.42) --
	( 59.77, 31.17);

\path[draw=drawColor,line width= 0.6pt,line join=round] ( 78.10, 28.42) --
	( 78.10, 31.17);

\path[draw=drawColor,line width= 0.6pt,line join=round] ( 96.44, 28.42) --
	( 96.44, 31.17);

\path[draw=drawColor,line width= 0.6pt,line join=round] (114.78, 28.42) --
	(114.78, 31.17);
\end{scope}
\begin{scope}
\path[clip] (  0.00,  0.00) rectangle (126.47,112.02);
\definecolor{drawColor}{gray}{0.10}

\node[text=drawColor,anchor=base,inner sep=0pt, outer sep=0pt, scale=  0.73] at ( 41.43, 20.16) {$100$};

\node[text=drawColor,anchor=base,inner sep=0pt, outer sep=0pt, scale=  0.73] at ( 59.77, 20.16) {$300$};

\node[text=drawColor,anchor=base,inner sep=0pt, outer sep=0pt, scale=  0.73] at ( 78.10, 20.16) {$500$};

\node[text=drawColor,anchor=base,inner sep=0pt, outer sep=0pt, scale=  0.73] at ( 96.44, 20.16) {$700$};

\node[text=drawColor,anchor=base,inner sep=0pt, outer sep=0pt, scale=  0.73] at (114.78, 20.16) {$900$};
\end{scope}
\begin{scope}
\path[clip] (  0.00,  0.00) rectangle (126.47,112.02);
\definecolor{drawColor}{gray}{0.10}

\node[text=drawColor,anchor=base,inner sep=0pt, outer sep=0pt, scale=  0.96] at ( 78.10,  7.75) {$\alpha$};
\end{scope}
\begin{scope}
\path[clip] (  0.00,  0.00) rectangle (126.47,112.02);
\definecolor{drawColor}{RGB}{0,0,0}
\definecolor{fillColor}{RGB}{255,0,0}

\path[draw=drawColor,line width= 0.4pt,line join=round,line cap=round,fill=fillColor] ( 53.38, 79.37) rectangle ( 56.48, 82.46);
\end{scope}
\begin{scope}
\path[clip] (  0.00,  0.00) rectangle (126.47,112.02);
\definecolor{drawColor}{RGB}{0,0,0}
\definecolor{fillColor}{RGB}{255,0,0}

\path[draw=drawColor,line width= 0.4pt,line join=round,line cap=round,fill=fillColor] ( 54.93, 73.63) circle (  1.75);
\end{scope}
\begin{scope}
\path[clip] (  0.00,  0.00) rectangle (126.47,112.02);
\definecolor{drawColor}{RGB}{0,0,0}
\definecolor{fillColor}{RGB}{255,0,0}

\path[draw=drawColor,line width= 0.4pt,line join=round,line cap=round,fill=fillColor] ( 54.93, 64.15) --
	( 57.12, 66.34) --
	( 54.93, 68.53) --
	( 52.74, 66.34) --
	cycle;
\end{scope}
\begin{scope}
\path[clip] (  0.00,  0.00) rectangle (126.47,112.02);
\definecolor{drawColor}{RGB}{0,0,0}

\node[text=drawColor,anchor=base west,inner sep=0pt, outer sep=0pt, scale=  0.69] at ( 59.38, 78.07) {\cite{Kummerle-NeurIPS2021} ($\ell_{1}$, \ref{eq:update-Kummerle})};
\end{scope}
\begin{scope}
\path[clip] (  0.00,  0.00) rectangle (126.47,112.02);
\definecolor{drawColor}{RGB}{0,0,0}

\node[text=drawColor,anchor=base west,inner sep=0pt, outer sep=0pt, scale=  0.69] at ( 59.38, 70.79) {\cite{Kummerle-NeurIPS2021} ($\ell_{0.5}$, \ref{eq:update-Kummerle})};
\end{scope}
\begin{scope}
\path[clip] (  0.00,  0.00) rectangle (126.47,112.02);
\definecolor{drawColor}{RGB}{0,0,0}

\node[text=drawColor,anchor=base west,inner sep=0pt, outer sep=0pt, scale=  0.69] at ( 59.38, 63.50) {\cite{Kummerle-NeurIPS2021} ($\ell_{0.1}$, \ref{eq:update-Kummerle})};
\end{scope}
\end{tikzpicture} \label{fig:sensitivity-alpha} }
	\caption{Figure \ref{fig:L1RR}: The relative error $\| \bx^{(t)}-\bx^*\|_2 / \|\bx^*\|_2$ at each iteration $t$ of different IRLS variants for $\ell_1$-regression ($k=200$). Figure \ref{fig:LpRR}: The relative error of $\texttt{IRLS}_p$ Algorithm \ref{algo:IRLS} and that of \cite{Daubechies-CPAM2010} for $\ell_p$-regression ($50$ iterations). Figure \ref{fig:sensitivity-alpha}: The sensitivity of \eqref{eq:update-Daubechies} and \eqref{eq:update-Kummerle} to mis-specification of $\alpha$ ($50$ iterations). In Figure \ref{fig:RR1}, we set $m=1000, n=10$, and results are averaged over $20$ trials. \label{fig:RR1}}
\end{figure}

\myparagraph{Dynamic Smoothing Parameter} 
Researchers have used different insights to reach the consensus of dynamically updating $\epsilon^{(t)}$ \cite{Daubechies-CPAM2010,Mukhoty-AISTATS2019,Aravkin-arXiv2019,Kummerle-NeurIPS2021,Yang-arXiv2021Lp}. The first insight is that convergence to global minimizers ensues if $\epsilon^{(t)}$ is suitably decreased to $0$ \cite{Daubechies-CPAM2010}. With $\beta\in(0,1)$, one such decreasing rule is
\begin{align}\label{eq:exponential-decay}
	\text{$\epsilon^{(t+1)}\gets \beta \epsilon^{(t)}$ if certain conditions are satisfied, or keep $\epsilon^{(t+1)} \gets \epsilon^{(t)}$ otherwise \cite{Mukhoty-AISTATS2019,Yang-arXiv2021Lp}}.
\end{align}
Figure \ref{fig:L1RR} depicts the performance of \cite{Mukhoty-AISTATS2019} with an arbitrary choice $\beta=0.5$ and $\epsilon^{(0)}=1$. Also shown in Figure \ref{fig:L1RR} are the IRLS methods with update rules of \cite{Daubechies-CPAM2010} and \cite{Kummerle-NeurIPS2021}, which we discuss next.



The basis pursuit paper \cite{Daubechies-CPAM2010} also proposed a dynamic updating rule for $\epsilon^{(t)}$. In our context \eqref{eq:Lp}, it sets
\begin{align}\label{eq:update-Daubechies}
		\epsilon^{(0)}\gets\infty, \ \ \ \epsilon^{(t+1)}\gets \min\big\{ \epsilon^{(t)}, [\br^{(t+1)}]_{\alpha+1}/m \big\},\ \ \ \ \ \ \ \br^{(t+1)}:= \bA \bx^{(t+1)}- \by\in\bbR^m,
\end{align}
where the hyper-parameter\footnote{In \cite{Daubechies-CPAM2010} and \cite{Kummerle-NeurIPS2021}, $\alpha$ is the sparsity level $k$, but $k$ might be unknown in practice, so we treat it as a hyper-parameter. That being said, in the experiments we set $\alpha=k$ by default, for simplicity and in light of Figure \ref{fig:sensitivity-alpha}.  \label{footnote:alpha} } $\alpha$ is a non-negative integer, and $[\br^{(t+1)}]_{\alpha+1}$ is the  $(\alpha+1)$-th largest element of the residual $\br^{(t+1)}$ in absolute values. Clearly, \eqref{eq:update-Daubechies} creates a non-increasing sequence of smoothing parameters $\epsilon^{(t)}$, and the decay rate of $\epsilon^{(t)}$ is adaptive to the data (in this case to the residual $\br^{(t+1)}$). 
While at first glance it is not clear how fast $\epsilon^{(t)}$ decays, \cite{Daubechies-CPAM2010} showed that the decay rate of $\epsilon^{(t)}$ in \eqref{eq:update-Daubechies} is locally linear for $p=1,\alpha=k$ under mild conditions, in accordance with the decay of the objective values. 

The final dynamic update rule for the parameter $\epsilon^{(t)}$ is from \cite{Kummerle-NeurIPS2021}\footref{footnote:alpha} and improves upon \eqref{eq:update-Daubechies} \cite{Daubechies-CPAM2010} via:
\begin{equation}\label{eq:update-Kummerle}
	\begin{split}
		\epsilon^{(0)} \gets\infty,\ \ \ \  \ \ \ \epsilon^{(t+1)}\gets \min\big\{ \epsilon^{(t)}, \sigma^{(t+1)}/m \big\} \ \ \ \ \ \ \ \ \ \ \ \ \ \ \\
		\text{where} \ \ \sigma^{(t+1)} \gets \min\Big\{ \norm{\br^{(t+1)} - \bz}{1}: \textnormal{$\bz\in\bbR^m$ is $\alpha$-sparse} \Big\}
	\end{split}
\end{equation}
Observe that \eqref{eq:update-Kummerle} computes the $\ell_1$-norm $\sigma^{(t+1)}$ of the best $\alpha$-term approximation of the residual $\br^{(t+1)}$, which is in general larger than  $[\br^{(t+1)}]_{\alpha+1}$ of \eqref{eq:update-Daubechies}, while both update rules \eqref{eq:update-Daubechies} and \eqref{eq:update-Kummerle} rely on the hyper-parameter $\alpha$. Ideally, if $\alpha=k$, and if IRLS with \eqref{eq:update-Daubechies} or \eqref{eq:update-Kummerle} converges to $\bx^*$, then both $[\br^{(t+1)}]_{\alpha+1}$ and $\sigma^{(t+1)}$ will approach $0$ for large enough $t$. While update rules \eqref{eq:update-Daubechies} and \eqref{eq:update-Kummerle} perform similarly for $\ell_1$-regression (Figure \ref{fig:L1RR}), \eqref{eq:update-Daubechies} fails more easily for small $p$ than \eqref{eq:update-Kummerle} (Figure \ref{fig:LpRR}). Finally, we note that overestimating the sparsity level $k$ by setting $\alpha > k$ deteriorates the performance of \eqref{eq:update-Kummerle} only slightly (Figure \ref{fig:sensitivity-alpha}).

\myparagraph{Summary} In Section \ref{subsection:IRLS-basic} we delivered two take-away messages: (1) IRLS with a fixed $\epsilon^{(t)}$ finds some approximate solution, sometimes good enough, (2) dynamically updating $\epsilon^{(t)}$ as per \eqref{eq:exponential-decay}-\eqref{eq:update-Kummerle} leads to local \cite{Daubechies-CPAM2010} or global \cite{Mukhoty-AISTATS2019,Kummerle-NeurIPS2021} linear convergence guarantees. In view of Figure \ref{fig:RR1}, we will next consider the $\texttt{IRLS}_p$ Algorithm \ref{algo:IRLS} with update rule \eqref{eq:update-Kummerle} for $\epsilon^{(t)}$, and analyze its convergence rates. 

\section{Convergence Theory of IRLS for Robust Regression}\label{section:IRLS-theory}
In Section \ref{subsection:RSP} we introduce the stable range space property and justify it as our core assumption. Under this assumption, we prove the global linear convergence (Theorem \ref{theorem:GLC-L1}) and local superlinear convergence (Theorem \ref{theorem:LSC-noiseless}) of the $\texttt{IRLS}_p$ Algorithm \ref{algo:IRLS} in Sections \ref{subsection:GLC-noiseless} and \ref{subsection:LSC-noiseless}, respectively. 

\subsection{The Stable Range Space Property}\label{subsection:RSP}
\begin{definition}[Stable Range Space Property]\label{def:RSP} A matrix $\bA\in\bbR^{m\times n}$ is said to satisfy the \emph{range space property (RSP)} of order $k$, or the $k$-RSP for short, if the following holds for any vector $\bd$ in the range space of $\bA$ and any set $S\subset\{1,\dots,m\}$ of cardinality at most $k$:
\begin{align}\label{eq:RSP}
	\sum_{i\in S} |d_i| <\sum_{i\in S^c} |d_i|
\end{align}
The \emph{stable RSP} of $\bA$ is defined in the same way as the RSP except that we now require
\begin{align}\label{eq:stable-RSP}
	\sum_{i\in S} |d_i| \leq \eta \sum_{i\in S^c} |d_i|
\end{align}
for some $\eta\in(0,1)$. We write ``$(k,\eta)$-stable RSP'' to emphasize the parameters $k$ and $\eta$.
\end{definition}
By definition, the stable RSP implies the RSP. Note that the $(k,\eta)$-stable RSP was proposed in \cite{Zhao-SIAM-J-O2012} to analyze a reweighted $\ell_1$-minimization algorithm for compressed sensing. With the notation of Proposition \ref{prop:LpRR=LpBP}, we can see that checking \eqref{eq:RSP} for all $\bd$ in the range space of $\bA$ is equivalent to checking it for all $\bd$ in the nullspace of $\bD$, the latter being the well-known \textit{nullspace property} (NSP) \cite{Cohen-JAMS2009,Foucart-2013}. In other words, the (stable) RSP of $\bA$ is equivalent to the (stable) NSP of $\bD$.

In Sections \ref{subsection:GLC-noiseless} and \ref{subsection:LSC-noiseless}, we will use the $(k,\eta)$-stable RSP as an assumption for analysis. Arguably, this is a weak assumption to make as the (stable) RSP is very close to a sufficient and necessary condition for exact recovery of the coefficient vector via $\ell_p$-robust regression \eqref{eq:Lp}:
\begin{prop}[Exact Recovery $\Leftrightarrow$ RSP]\label{prop:RSP=exact-recovery}
Let $p\in(0,1]$. For all $\bx^*$ and $\by$ such that $\bA \bx^*-\by$ is $k$-sparse $\bx^*$ is the unique solution to \eqref{eq:Lp} if and only if $\bA\in\bbR^{m\times n}$ satisfies the $k$-RSP.
\end{prop}

It follows from \cite[Theorem 5]{Tillmann-TIT2013} and Proposition \ref{prop:LpRR=LpBP} that checking whether a given matrix $\bA$ satisfies the stable RSP is \textit{co-NP-complete} \cite{Arora-2009}. On the other hand, we show that random Gaussian matrices of size $m\times n$ with sufficiently large $m$ satisfy the stable RSP with high probability, which further justifies the usage of the stable RSP as an assumption in our analysis of $\texttt{IRLS}_p$:
\begin{prop}[Gaussian $\Rightarrow$ RSP]\label{prop:Gaussian-RSP}
	Suppose $m-n \geq 2k$. Let $\bA\in\bbR^{m\times n}$ be a matrix with i.i.d. $\cN(0,1)$ entries. Let $\delta\in (0,1)$ and $\eta\in(0,1]$ be fixed constants. If it holds that
	\begin{align}\label{eq:condition-Gaussian-RSP}
		\frac{(m-n)^2}{m-n+1} \geq 2k \ln(em/k) \cdot \bigg( 1.67 +   \eta^{-1}  + \frac{\sqrt{18\ln (2.5\delta^{-1})}}{\sqrt{2k \ln(em/k)}}  \bigg)^2,
	\end{align}
	then $\bA$ satisfies the  $(k,\eta)$-stable RSP with probability at least $1-\delta$.
\end{prop}
The assumption $m-n \geq 2k$ of Proposition \ref{prop:Gaussian-RSP} is necessary, as it is not hard to prove that, if $m-n < 2k$ and $\bA\bx^*-\by$ is $k$-sparse, then the $\ell_0$-minimization problem $\min_{x\in\bbR^n} \|\bA \bx -\by \|_0$ has multiple solutions. With this assumption, \eqref{eq:condition-Gaussian-RSP} roughly becomes $m-n\geq ck \log(em/k)$ for some constant $c$ if $m-n$ is large enough; thus, ignoring logarithmic and constant factors, condition \eqref{eq:condition-Gaussian-RSP} becomes $m-n\geq \Theta(k)$, which is nearly optimal, as the condition $m-n \geq 2k$ is necessary. 

Proposition \ref{prop:RSP=exact-recovery} follows directly from \cite[Theorem 4.9]{Foucart-2013}, so we omit its proof. The reader might also find Proposition \ref{prop:Gaussian-RSP} corresponds to \cite[Corollary 9.34]{Foucart-2013}. This corollary assumes $\bD$ of \eqref{eq:Lp-BP} has i.i.d. $\cN(0,1)$ entries, from which is not immediate what the distribution of $\bA$ is, thus it does not imply Proposition \ref{prop:Gaussian-RSP} directly; this is why we provide a complete proof for Proposition \ref{prop:Gaussian-RSP} in Appendix \ref{appendix:RSP-Gaussian}.

\subsection{Global Linear Convergence of $\texttt{IRLS}_1$}\label{subsection:GLC-noiseless}
With the background provided in Section \ref{section:background}, we are now ready to state the first main result:
\begin{theorem}[Global Linear Convergence]\label{theorem:GLC-L1} Suppose $\bA\in\bbR^{m\times n}$ obeys the $(k,\eta)$-stable RSP with $\eta\in(0,3/4)$ (cf. Definition \ref{def:RSP}). Let $\bA \bx^* - \by$ be $k$-sparse. If the smoothing parameter $\epsilon^{(t)}$ is updated as per \eqref{eq:update-Kummerle} with $\alpha=k$ during the execution of the $\texttt{IRLS}_1$ Algorithm \ref{algo:IRLS}, then it holds for any $t\geq 1$ that
	\begin{align}\label{eq:linear-convergence}
	\norm{\bA\bx^{(t+1)} - \by}{1} - \norm{\bA\bx^{*}-\by}{1} \leq \bigg(1- \frac{(3-4\eta)^2}{294\eta m} \bigg)^t \cdot 3\cdot\norm{\bA \bx^{(1)}-\by}{1}, 
	\end{align}
	meaning that the $\texttt{IRLS}_1$ Algorithm \ref{algo:IRLS} converges linearly and globally in objective value.
\end{theorem}

\myparagraph{Discussion} The condition $\eta\in(0,3/4)$ is better understood via Proposition \ref{prop:Gaussian-RSP}, which asserts that this condition holds with high probability if $\bA$ has i.i.d. Gaussian entries and if $m$ is large enough (compared to $n,k$). Moreover, if  the $(k,\eta)$-stable RSP holds, Proposition \ref{prop:RSP=exact-recovery} implies that $\bx^*$ is a global minimizer of \eqref{eq:Lp} with $p=1$. The strength of Theorem \ref{theorem:GLC-L1} is in that it guarantees linear convergence for \textit{any initialization, starting at the first iteration}. 
However, the price to pay is that the
convergence rate seems conservative with constant $1/294$, and depends on the number $m$ of samples: \eqref{eq:linear-convergence} predicts that $\texttt{IRLS}_1$ converges with accuracy $\delta$ in $O(m\log (1/\delta))$ iterations and $m$ can be large. We argue that such dependency is not an artifact of our analysis, because similar dependencies have been established for IRLS variants for $\ell_p$-regression \eqref{eq:Lp} with $p>2$ \cite[Theorem 3.1]{Adil-NeurIPS2019}, for interior point methods in convex optimization with $m$ inequality constraints (\cite[Chapter 11]{Boyd-book2004}, \cite[Chapter 5]{Nesterov-2018}), and for SGD in smooth \& strongly convex optimization with a sum of $m$ function components \cite[Theorem 2.1]{Allen-JMLR2017}). Moreover, an adversarial initialization for which the rate empirically depends on $m$ was pointed out for IRLS for basis pursuit in \cite{Kummerle-NeurIPS2021}. That being said, we conjecture that the linear dependency on $m$ in \eqref{eq:linear-convergence} can be improved to $\sqrt{m}\log m$ for some different choice of $\epsilon^{(t)}$; see \cite{Adil-NeurIPS2019} and \cite[Sections 11.5 and 11.8.2]{Boyd-book2004} for why this conjecture makes sense. Furthermore, Figure \ref{fig:L1RR} empirically depicts that, with the least-squares initialization, i.e., with $w_i^{(0)}=1$ for all $i$, $\texttt{IRLS}_1$ typically converges to $\bx^*$ in $30$ iterations. This hints to the possibility of alleviating the dependency on $m$ by analyzing the least-squares initialization, a challenging task which we leave to future work.

\myparagraph{Comparison to \cite{Mukhoty-AISTATS2019}} To our knowledge, \cite{Mukhoty-AISTATS2019} is the only paper that claimed the global linear convergence of an IRLS variant with $\epsilon^{(t)}$ updated as per \eqref{eq:exponential-decay} and with $p=1$. In particular, \cite{Mukhoty-AISTATS2019} sets $\epsilon^{(t+1)}\gets \beta \epsilon^{(t)}$ whenever $\| \bx^{(t+1)} - \bx^{(t)} \|_2\leq 2\beta \epsilon^{(t)}$; see Figure \ref{fig:L1RR} for the convergence of IRLS under this rule. The update rule $\epsilon^{(t+1)}\gets \beta \epsilon^{(t)}$ of \cite{Mukhoty-AISTATS2019} has the advantage of being faster to compute than \eqref{eq:update-Kummerle}. In spite of this, and of their beautiful proof idea, their result has several disadvantages compared to Theorem \ref{theorem:GLC-L1}. First, their proof involves sophisticated probabilistic arguments, and is tailored towards a feature matrix that satisfies strong concentration properties. Our analysis, on the other hand, uses the $(k,\eta)$-stable RSP, which is close to a necessary and sufficient condition for the success of $\ell_1$-regression (cf. Proposition \ref{prop:RSP=exact-recovery}) and expected to hold for a much larger range of matrices than just (sub-)Gaussian matrices (see \cite{LecueMendelson-EMS17,DirksenLecueRauhut-IEEEIT2017} for related results in the context of basis pursuit).  Second, some statements in their proof are inaccurate (e.g., their Lemma 5 does not hold for $k<n$, the last paragraph in their proving Lemma 8 is not rigorous). Third, their update rule incurs two hyper-parameters, $\epsilon^{(0)}$ and $\beta$, while Theorem \ref{theorem:GLC-L1} is based on the decreasing rule \eqref{eq:update-Kummerle} that is adaptive to data, and the only hyper-parameter $\alpha$ is easier to set (Figure \ref{fig:sensitivity-alpha}).

\myparagraph{Connection to \cite{Kummerle-NeurIPS2021}} 
A related global linear convergence was established for the IRLS method for basis pursuit in \cite[Theorem 3.2]{Kummerle-NeurIPS2021}. Our proof is inspired by \cite{Kummerle-NeurIPS2021}, but also improves on its strategy. For example, their Theorem 3.2 involves two parameters for the \textit{stable nullspace property}, while we only have a single RSP parameter $\eta$, simplifying matters. Also, we obtain a constant of $1/294$ in Theorem \ref{theorem:GLC-L1}, which is better than the value of $1/768$ for the respective constant in \cite[Theorem 3.2]{Kummerle-NeurIPS2021}.

It is important to note that the proofs of \cite{Mukhoty-AISTATS2019,Kummerle-NeurIPS2021} and ours heavily rely on certain rules for decreasing the smoothing parameter $\epsilon^{(t)}$; all these proofs of global linear convergence for IRLS would break down if $\epsilon^{(t)}$ were fixed. While the main theorems of both \cite{Mukhoty-AISTATS2019} and \cite{Kummerle-NeurIPS2021} are limited to the $p=1$ case, we present our Theorem \ref{theorem:LSC-noiseless} for $p\in(0,1)$ next.

\subsection{Local Superlinear Convergence of $\texttt{IRLS}_p$}\label{subsection:LSC-noiseless}
The $\ell_p$-regression problem \eqref{eq:Lp} with $p\in(0,1)$ is more challenging than the case $p=1$ due to the lack of convexity. However, here we show that at least locally,  $\texttt{IRLS}_p$ converges with a superlinear rate of order $2-p$. Moreover, since it is valid to run $\texttt{IRLS}_p$ even with $p=0$, we carefully design the proof such that the following result holds not only for $p\in(0,1]$, but also for $p=0$, in which case $\texttt{IRLS}_p$ can be interpreted as an algorithm minimizing a sum-of-logarithm objective (see Appendix \ref{sec:IRLS0:MM}).
\begin{theorem}[Local Superlinear Convergence]\label{theorem:LSC-noiseless}
	Run $\texttt{IRLS}_p$ with $p\in[0,1]$ and update $\epsilon^{(t)}$ by \eqref{eq:update-Kummerle} with $\alpha=k$. Assume $\bA$ satisfies the $(k,\eta)$-stable RSP (Definition \ref{def:RSP}). Let $c\in(0,1)$ be a sufficiently small constant such that $2 c^{1-p}\eta (\eta+1)<(1-c)^{2-p}$. Let $\bA \bx^*- \by$ be $k$-sparse with support $S^*$. Define
	\begin{align}\label{eq:mu}
		\mu := 2\eta (\eta+1)  (1-c)^{p-2} \cdot \min_{i\in S^*} |\trsp{\ba_i}\bx^* - y_i|^{p-1}.
	\end{align}
	If the initialization $\bx^{(1)}$ is in a neighborhood of $\bx^*$, in the sense that 
	\begin{align}\label{eq:Lp-convergence-radius}
		\norm{\bA \bx^{(1)} - \bA \bx^* }{1} \leq c\cdot \min_{i\in S^*} |\trsp{\ba_i}\bx^* - y_i|, 
	\end{align}
	then $\texttt{IRLS}_p$ achieves the following superlinear convergence rate of order $2-p$ for every $t\geq 1$:
	\begin{align}\label{eq:Lp-rates}
		\norm{\bA \bx^{(t+1)} - \bA \bx^* }{1} \leq \mu\cdot \Big( \norm{\bA \bx^{(t)} - \bA \bx^* }{1} \Big)^{2-p} < \norm{\bA \bx^{(t)} - \bA \bx^* }{1}
	\end{align}
	In particular, for $p\in[0,1)$, the superlinear rate \eqref{eq:Lp-rates} implies
	\begin{align}\label{eq:Lp-rates2}
		\norm{\bA \bx^{(t+1)} - \bA \bx^* }{1} \leq \Big( \mu^{1/(1-p)} \cdot \norm{\bA \bx^{(1)} - \bA \bx^* }{1} \Big)^{ (2-p)^t-1 } \cdot \norm{\bA \bx^{(1)} - \bA \bx^* }{1}.
	\end{align}
\end{theorem}
\begin{example}[\textit{Quadratic versus linear rates}]
	We illustrate the practical benefit of a superlinear (quadratic) convergence rate with the following simple calculation. Suppose that the inequality in \eqref{eq:Lp-convergence-radius} is barely fulfilled such that in the case of $p=0$,  $\bx^{(1)}$ satisfies  $\mu\cdot\|\bA \bx^{(1)} - \bA \bx^* \|_1 \leq0.9999999$. Then \eqref{eq:Lp-rates2} implies that after $30$ iterations, the residual error is far below numerical precision already since $\|\bA \bx^{(31)} - \bA \bx^* \|_1\leq 0.9999999^{2^{30}-1}\cdot \|\bA \bx^{(1)} - \bA \bx^* \|_1\approx 10^{-47} \cdot \|\bA \bx^{(1)} - \bA \bx^* \|_1$. On the other hand, for $p=1$, a linear convergence factor $\mu$ of $\mu\leq 0.9999999$ is only able to guarantee a decay of the order $\mu^{30}\leq 0.9999999^{30}\approx0.999997$ after $30$ iterations.
\end{example}
\myparagraph{Discussion} Even though $\texttt{IRLS}_p$ with $p\in(0,1)$ is tailored to $\ell_p$-regression, we measure the progress of $\texttt{IRLS}_p$ in $\ell_1$-norm, e.g., we provide upper bounds on $\| \bA \bx^{(t+1)} - \bA \bx^*\|_1$; this is because doing so avoids the use of H\"older's inequality, which allows us to give tighter results that improve on \cite{Daubechies-CPAM2010} (see the next paragraph). The second point that deserves discussion is the local convergence neighborhood defined by the right-hand side of \eqref{eq:Lp-convergence-radius}. Note that, we typically assume that every outlier sample $(\ba_i,y_i)$ (with $i\in S^*$) would result in a relatively large residual at $\bx^*$, say $|\trsp{\ba_i}\bx^* - y_i|\gg 0$. Hence the minimization term of the right-hand side in \eqref{eq:Lp-convergence-radius} is well-behaved. Next, one might ask how to obtain such an initialization $\bx^{(1)}$ that satisfies \eqref{eq:Lp-convergence-radius}. One might run $\texttt{IRLS}_1$ to produce such an initialization; note though that the upper bound of \eqref{eq:Lp-convergence-radius} can not be computed, so one could not detect when to switch from $\texttt{IRLS}_1$ to $\texttt{IRLS}_p$ ($p\in[0,1)$). Or alternatively, one could run $\texttt{IRLS}_p$ with a least-squares initialization and count on empirical global (super)linear convergence of $\texttt{IRLS}_p$; the latter is what we did in the experiments and is what we recommend. A final and important remark is that $\ell_p$-regression \eqref{eq:Lp} with $p\in(0,1)$ is in general NP-hard \cite[Exercise 2.10]{Foucart-2013}, but this does not contradict the theoretical local superlinear convergence of Theorem \ref{theorem:LSC-noiseless} and the empirical global convergence in Figures \ref{fig:LpRR} and \ref{fig:SLR1}, and this does not mean that $\texttt{IRLS}_p$ solves an NP-hard problem in polynomial time. The catch is that we operate under the assumption that $\bx^*$ leads to a $k$-sparse residual and is a unique global minimizer of \eqref{eq:Lp} (cf. Proposition \ref{prop:RSP=exact-recovery}). With this assumption and small enough $k/m$, $\ell_p$-regression is tractable and can be solved via $\texttt{IRLS}_p$ to high accuracy (Figure \ref{fig:LpRR}).




\myparagraph{Connection to \cite{Daubechies-CPAM2010}} \cite[Theorem 7.9]{Daubechies-CPAM2010} proves the local superlinear convergence of IRLS with update rule \eqref{eq:update-Daubechies} for $\ell_p$-basis pursuit \eqref{eq:Lp-BP}, which motivates Theorem \ref{theorem:LSC-noiseless}. However, \cite[Theorem 7.9]{Daubechies-CPAM2010} requires $c$ of \eqref{eq:Lp-convergence-radius} to satisfy $c = O(1/m^{2/p-1})$, which means that, as $p\to 0$ or $m\to \infty$, the eligible value of $c$ quickly becomes vanishingly small, and thus the radius of local convergence \eqref{eq:Lp-convergence-radius} diminishes greatly. In contrast, in our condition $2 c^{1-p}\eta (\eta+1)<(1-c)^{2-p}$, $c$ has no direct dependency on $m$ and is well-behaved even if $p\to 0$. The above difference is due partly to our improved proof strategy, and partly to the different update rules of the smoothing parameter $\epsilon^{(t)}$. Figure \ref{fig:LpRR} showed that rule \eqref{eq:update-Daubechies} of \cite{Daubechies-CPAM2010} does not work well for small $p$, in contrast to rule \eqref{eq:update-Kummerle} that we use; a possible explanation for this phenomenon is that rule \eqref{eq:update-Daubechies} decreases $\epsilon^{(t)}$ too fast, yielding a poor initialization for the next iteration of weighted least-squares (similarly to the interior point method, cf. \cite{Boyd-book2004}). Finally, \cite[Theorem 7.9]{Daubechies-CPAM2010} does not hold for $p=0$, for which Theorem \ref{theorem:LSC-noiseless} still holds and suggests a local quadratic rate.

\section{Applications and Experiments}\label{section:applications}
We now explore the performance of $\texttt{IRLS}_p$ in three different applications, \textit{real phase retrieval} (Section \ref{subsection:RPR}), \textit{linear regression without correspondences} (Section \ref{subsection:SLR}), and \textit{face restoration} (Section \ref{subsection:face}). Note that the first two applications are special examples of the recent algebraic-geometric framework called  \textit{homomorphic sensing} \cite{Tsakiris-ICML2019}, \cite[Section 4]{Peng-ICML2021}, \cite[Section 1.2]{Peng-ACHA2021}. In Section \ref{subsection:choice-p}, we examine the behavior of $\texttt{IRLS}_p$ for different values of $p$ and under noise.

\subsection{Real Phase Retrieval}\label{subsection:RPR}
In real phase retrieval \cite{Chen-MP2019,Tan-AA-J-IMA2019}, we are given a measurement matrix $\bA=\trsp{[\ba_1,\dots,\ba_m]}$ and $\by=\trsp{[y_1,\dots,y_m]}$, where $y_i:=|\trsp{\ba_i} \bx^*|$, and we need to find either $\bx^*$ or $-\bx^*$. This problem is a relative of the complex phase retrieval problem \cite{Balan-ACHA2006}, where all data $\by$ and $\bA$, as well as the ground-truth $\bx^*$, are complex-valued, and which has applications in X-ray crystallography \cite{Grohs-SIAM-Review2020}.

Here we show that real phase retrieval can be solved via $\ell_p$-regression \eqref{eq:Lp}. Consider the index sets
\begin{align}\label{eq:I}
	I^+:=\{i:\trsp{\ba_i} \bx^* > 0\}, \ \ \ \ \ \ \  I^{-}:=\{i:\trsp{\ba_i} \bx^* < 0\} .
\end{align} 
We might assume $\trsp{\ba_i} \bx^* \neq 0$ for every $i$ without loss of generality. Then we know that either $\bA \bx^*-\by$ has its $\ell_0$ norm equal to $m-|I^{-}|$, or $\bA (-\bx^*)-\by$ has its $\ell_0$ norm equal to $m-|I^{+}|$. With such sparsity patterns on these two residuals, we can minimize \eqref{eq:Lp}, which serves as a (non-convex) relaxation of $\ell_0$-minimization, to recover $\bx^*$ or $-\bx^*$, whichever corresponds to a sparser residual. Here, we essentially treat one of the two clusters defined by \eqref{eq:I} as inliers, and the other as outliers. This robust regression point of view on real phase retrieval seems to be known by experts \cite[Section 1.2]{Huang-arXiv2021b}, but we have not found a paper that actually proposes to solve \eqref{eq:Lp} for real phase retrieval. 

Here we argue with Figure \ref{fig:RPR} that, solving \eqref{eq:Lp} (by $\texttt{IRLS}_{0.1}$) for real phase retrieval can be beneficial. Indeed, theoretically, $\texttt{IRLS}_{0.1}$ converges to $\pm\bx^*$ locally but superlinearly (as a corollary of Theorem~\ref{theorem:LSC-noiseless}). Empirically, Figure \ref{fig:RPR} shows that $\texttt{IRLS}_{0.1}$ succeeds in recovering $\pm\bx^*$ by using only $m=2n-1$ samples, which is exactly the theoretical minimum for real phase retrieval \cite[Proposition 2.5]{Balan-ACHA2006}. Figure \ref{fig:RPR} also shows that multiple other state-of-the-art methods\footnote{Reviewing these algorithms is beyond the scope of this paper. Some of them are designed for complex phase retrieval but is applicable to the real case. We used the implementations from PhasePack \cite{Chandra-SampTA2019}.} \cite{Tan-AA-J-IMA2019,Wei-IP2015,Dhifallah-Allerton2017,Chen-NIPS2015,Zeng-arXiv2017} fail to recover $\pm \bx^*$ in this extreme situation, even though some of them have \textit{nearly optimal} sample complexity, which requires roughly $O(n)$ samples up to a logarithmic factor to succeed (\textit{caution}: nearly optimal $\neq$ optimal). As suggested by Proposition \ref{prop:Gaussian-RSP}, for $|I^+|$ fixed, minimizing \eqref{eq:Lp} with a Gaussian matrix $\bA$ enjoys $m=O(n)$ sample complexity, up to also a logarithmic factor.

Despite these advantages, we also emphasize that solving \eqref{eq:Lp} (via $\texttt{IRLS}_{0.1}$) for real phase retrieval has an inherent limit: The performance depends on the number $|I^+|$ of positive signs and the number $m$ of samples; we might call $|I^+|/m$ the inlier rate or outlier rate. At $n/m=200/399$, we see that $\texttt{IRLS}_{0.1}$ succeeds if $|I^+|\leq 70$, and if $|I^+| \geq 399-70$ by symmetry, but it fails if $|I^+|$ and $|I^-|$ get closer (Figure \ref{fig:RPR}). For $\texttt{IRLS}_{0.1}$ to successfully handle the case $|I^+|=|I^-|$, we need more samples, empirically, say $m\geq 5n$. It is an interesting future direction to design an algorithm that can handle the case of minimum samples ($m=2n-1$) and balanced data ($|I^+|= |I^-|$).

\begin{figure}
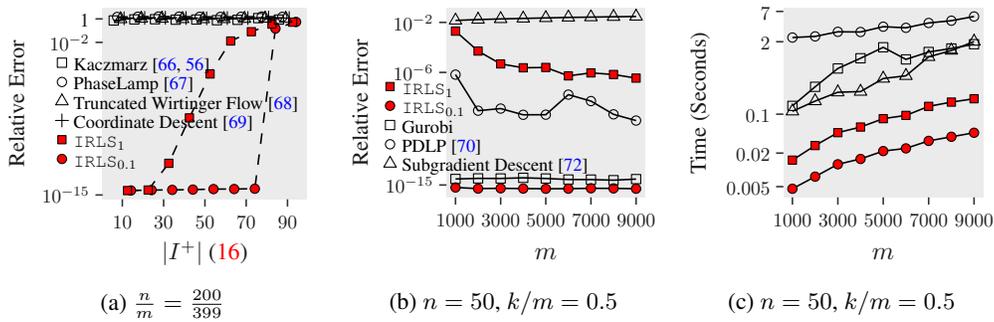

	\centering
	\subfloat[$\frac{n}{m}=\frac{200}{399}$]{\input{./figures/experiment_RPR_minimum_samples.tex} \label{fig:RPR} } 
	\subfloat[$n=50$, $k/m=0.5$]{\input{./figures/experiment_SLR_error_m.tex} \label{fig:SLR1} }
	\subfloat[$n=50$, $k/m=0.5$]{\input{./figures/experiment_SLR_time_m.tex} \label{fig:SLR2} }
	\caption{Figure \ref{fig:RPR}: Relative error $\min\{ \| \hat{\bx}- \bx^*\|_2, \| \hat{\bx} + \bx^*\|_2 \} /\| \bx^*\|_2$ of the methods that produce estimates $\hat{\bx}$ for real phase retrieval. Figures \ref{fig:SLR1}-\ref{fig:SLR2}: Relative errors $\| \hat{\bx}- \bx^*\|_2 /\| \bx^*\|_2$ and running times for linear regression without correspondences. $\texttt{IRLS}_p$ run for at most $50$ iterations. $50$ trials. \label{fig:application}}
\end{figure}

\subsection{Linear Regression without Correspondences}\label{subsection:SLR}
Compared to (real) phase retrieval, the problem of \textit{linear regression without correspondences} is an also important, but less developed subject; see \cite{Unnikrishnan-TIT18,Pananjady-TIT18,Slawski-JoS19,Tsakiris-TIT2020,Peng-SPL2020,Slawski-JCGS2021,Xie-ICLR2021,Li-ICCV2021} for recent advances. In this problem we are given $\by\in\bbR^m$ and $\bA\in\bbR^{m\times n}$, and we need to solve the equations
\begin{align}\label{eq:SLR}
	\by = \bPi \bA \bx 
\end{align}
for some unknown permutation matrix $\bPi\in\bbR^{m\times m}$ and vector $\bx\in\bbR^n$, under the assumption that there exists a solution $(\bPi^*,\bx^*)$. Solving \eqref{eq:SLR} is in general NP-hard for $n>1$ \cite{Pananjady-TIT18,Hsu-NIPS17}. However, if we further assume that $\bPi^*$ permutes at most $k$ rows of $\bA$, then we know that the residual $\bA \bx^* - \by$ is $k$-sparse. This is an insight of \cite{Slawski-JoS19}, where it was proposed to estimate $\bx^*$ by solving \eqref{eq:Lp} with $p=1$.

We show that $\texttt{IRLS}_p$ with $p\in(0,1)$ is more suitable than several baselines for solving this robust regression problem reduced from linear regression without correspondences. As baselines, we use the PDLP \cite{Applegate-NeurIPS2021} and Gurobi \cite{Gurobi951}  solvers to solve \eqref{eq:Lp} with $p=1$ as a linear program, and we also use a subgradient descent method implemented in the FOM toolbox of Beck \& Guttmann-Beck \cite{Beck-OMS2019}. Figure \ref{fig:SLR1} shows that $\texttt{IRLS}_{0.1}$ is the most accurate, and Figure \ref{fig:SLR2} shows that $\texttt{IRLS}_{0.1}$ is $30$ times faster than Gurobi and subgradient descent, and is more than $100$ times faster than PDLP. Interestingly, $\texttt{IRLS}_{1}$ is not very competitive in terms accuracy (Figure \ref{fig:SLR1}). Moreover, $\texttt{IRLS}_1$ is slower than $\texttt{IRLS}_{0.1}$, as $\texttt{IRLS}_{0.1}$ converges faster and thus terminates earlier than $\texttt{IRLS}_1$; we terminate $\texttt{IRLS}_{p}$ whenever the objective stops decreasing (up to a tolerance $10^{-15}$). Finally, see also \cite{Tsakiris-JMLR2018,Aftab-WCACV2015,Dong-PRL2019,Iwata-ECCV2020,Ding-CVPR2020,Yang-RA-L2020} where their IRLS variants outperform a different set of baselines (for different problems).


\subsection{Face Restoration from Sparsely Corrupted Measurements}\label{subsection:face}
Here we consider a simple face restoration experiment on the Extended Yale B dataset \cite{Georghiades-PAMI2001}, downsampled as per \cite{You-CVPR2016}. This dataset contains the face images of $38$ individuals under about $n\approx60$ different illuminations. Let $\bF=[\bm{f}_1,\dots,\bm{f}_{n+1}]$ be the matrix of faces of the same individual, where each $\bm{f}_i\in\mathbb{R}^{m}$ is a (vectorized) face image with $m=2,016$ pixels. We assume that $\bF$ is approximately low-rank (which is true under the \textit{Lambertian reflectance} assumption \cite{Basri-TPAMI2003}), thus each $\bm{f}_i$ can be approximately represented as a linear combination of other faces of the individual, i.e., $\bm{f}_i\approx\bF_{i} \bx_i^*$ for some $\bx_i^*\in\bbR^n$, where $\bF_{i}\in\bbR^{m\times n}$ is the same as $\bF$, except with the $i$-th column excluded.


\begin{figure}
	\centering
	\subfloat[Quantitative results]{
\begin{tikzpicture}[x=1pt,y=1pt]
\definecolor{fillColor}{RGB}{255,255,255}
\path[use as bounding box,fill=fillColor,fill opacity=0.00] (0,0) rectangle (122.86,108.41);
\begin{scope}
\path[clip] (  0.00,  0.00) rectangle (122.86,108.41);
\definecolor{drawColor}{RGB}{255,255,255}
\definecolor{fillColor}{RGB}{255,255,255}

\path[draw=drawColor,line width= 0.6pt,line join=round,line cap=round,fill=fillColor] ( -0.00,  0.00) rectangle (122.86,108.41);
\end{scope}
\begin{scope}
\path[clip] ( 37.48, 31.17) rectangle (117.36,102.90);
\definecolor{fillColor}{gray}{0.92}

\path[fill=fillColor] ( 37.48, 31.17) rectangle (117.36,102.91);
\definecolor{drawColor}{RGB}{0,0,0}

\path[draw=drawColor,line width= 0.6pt,line join=round] ( 41.11, 34.43) --
	( 53.21, 34.43) --
	( 65.32, 34.43) --
	( 77.42, 34.43) --
	( 89.52, 34.43) --
	(101.63, 34.43) --
	(113.73, 34.43);

\path[draw=drawColor,line width= 0.6pt,line join=round] ( 41.11, 40.62) --
	( 53.21, 49.03) --
	( 65.32, 58.37) --
	( 77.42, 68.22) --
	( 89.52, 78.48) --
	(101.63, 88.96) --
	(113.73, 99.64);

\path[draw=drawColor,line width= 0.6pt,line join=round] ( 41.11, 42.11) --
	( 53.21, 42.75) --
	( 65.32, 43.45) --
	( 77.42, 44.42) --
	( 89.52, 45.62) --
	(101.63, 47.05) --
	(113.73, 48.64);

\path[draw=drawColor,line width= 0.4pt,line join=round,line cap=round] ( 39.36, 32.68) rectangle ( 42.86, 36.18);

\path[draw=drawColor,line width= 0.4pt,line join=round,line cap=round] ( 51.46, 32.68) rectangle ( 54.96, 36.18);

\path[draw=drawColor,line width= 0.4pt,line join=round,line cap=round] ( 63.57, 32.68) rectangle ( 67.06, 36.18);

\path[draw=drawColor,line width= 0.4pt,line join=round,line cap=round] ( 75.67, 32.68) rectangle ( 79.17, 36.18);

\path[draw=drawColor,line width= 0.4pt,line join=round,line cap=round] ( 87.77, 32.68) rectangle ( 91.27, 36.18);

\path[draw=drawColor,line width= 0.4pt,line join=round,line cap=round] ( 99.88, 32.68) rectangle (103.37, 36.18);

\path[draw=drawColor,line width= 0.4pt,line join=round,line cap=round] (111.98, 32.68) rectangle (115.48, 36.18);

\path[draw=drawColor,line width= 0.4pt,line join=round,line cap=round] ( 41.11, 40.62) circle (  1.96);

\path[draw=drawColor,line width= 0.4pt,line join=round,line cap=round] ( 53.21, 49.03) circle (  1.96);

\path[draw=drawColor,line width= 0.4pt,line join=round,line cap=round] ( 65.32, 58.37) circle (  1.96);

\path[draw=drawColor,line width= 0.4pt,line join=round,line cap=round] ( 77.42, 68.22) circle (  1.96);

\path[draw=drawColor,line width= 0.4pt,line join=round,line cap=round] ( 89.52, 78.48) circle (  1.96);

\path[draw=drawColor,line width= 0.4pt,line join=round,line cap=round] (101.63, 88.96) circle (  1.96);

\path[draw=drawColor,line width= 0.4pt,line join=round,line cap=round] (113.73, 99.64) circle (  1.96);
\definecolor{fillColor}{RGB}{255,0,0}

\path[draw=drawColor,line width= 0.4pt,line join=round,line cap=round,fill=fillColor] ( 41.11, 42.11) circle (  1.96);

\path[draw=drawColor,line width= 0.4pt,line join=round,line cap=round,fill=fillColor] ( 53.21, 42.75) circle (  1.96);

\path[draw=drawColor,line width= 0.4pt,line join=round,line cap=round,fill=fillColor] ( 65.32, 43.45) circle (  1.96);

\path[draw=drawColor,line width= 0.4pt,line join=round,line cap=round,fill=fillColor] ( 77.42, 44.42) circle (  1.96);

\path[draw=drawColor,line width= 0.4pt,line join=round,line cap=round,fill=fillColor] ( 89.52, 45.62) circle (  1.96);

\path[draw=drawColor,line width= 0.4pt,line join=round,line cap=round,fill=fillColor] (101.63, 47.05) circle (  1.96);

\path[draw=drawColor,line width= 0.4pt,line join=round,line cap=round,fill=fillColor] (113.73, 48.64) circle (  1.96);
\end{scope}
\begin{scope}
\path[clip] (  0.00,  0.00) rectangle (122.86,108.41);
\definecolor{drawColor}{gray}{0.10}

\node[text=drawColor,anchor=base east,inner sep=0pt, outer sep=0pt, scale=  0.73] at ( 32.53, 29.72) {$0.1$};

\node[text=drawColor,anchor=base east,inner sep=0pt, outer sep=0pt, scale=  0.73] at ( 32.53, 53.14) {$0.2$};

\node[text=drawColor,anchor=base east,inner sep=0pt, outer sep=0pt, scale=  0.73] at ( 32.53, 76.55) {$0.3$};

\node[text=drawColor,anchor=base east,inner sep=0pt, outer sep=0pt, scale=  0.73] at ( 32.53, 92.95) {$0.37$};
\end{scope}
\begin{scope}
\path[clip] (  0.00,  0.00) rectangle (122.86,108.41);
\definecolor{drawColor}{gray}{0.20}

\path[draw=drawColor,line width= 0.6pt,line join=round] ( 34.73, 32.75) --
	( 37.48, 32.75);

\path[draw=drawColor,line width= 0.6pt,line join=round] ( 34.73, 56.17) --
	( 37.48, 56.17);

\path[draw=drawColor,line width= 0.6pt,line join=round] ( 34.73, 79.58) --
	( 37.48, 79.58);

\path[draw=drawColor,line width= 0.6pt,line join=round] ( 34.73, 95.98) --
	( 37.48, 95.98);
\end{scope}
\begin{scope}
\path[clip] (  0.00,  0.00) rectangle (122.86,108.41);
\definecolor{drawColor}{gray}{0.20}

\path[draw=drawColor,line width= 0.6pt,line join=round] ( 41.11, 28.42) --
	( 41.11, 31.17);

\path[draw=drawColor,line width= 0.6pt,line join=round] ( 65.32, 28.42) --
	( 65.32, 31.17);

\path[draw=drawColor,line width= 0.6pt,line join=round] ( 89.52, 28.42) --
	( 89.52, 31.17);

\path[draw=drawColor,line width= 0.6pt,line join=round] (113.73, 28.42) --
	(113.73, 31.17);
\end{scope}
\begin{scope}
\path[clip] (  0.00,  0.00) rectangle (122.86,108.41);
\definecolor{drawColor}{gray}{0.10}

\node[text=drawColor,anchor=base,inner sep=0pt, outer sep=0pt, scale=  0.73] at ( 41.11, 20.16) {$10\%$};

\node[text=drawColor,anchor=base,inner sep=0pt, outer sep=0pt, scale=  0.73] at ( 65.32, 20.16) {$30\%$};

\node[text=drawColor,anchor=base,inner sep=0pt, outer sep=0pt, scale=  0.73] at ( 89.52, 20.16) {$50\%$};

\node[text=drawColor,anchor=base,inner sep=0pt, outer sep=0pt, scale=  0.73] at (113.73, 20.16) {$70\%$};
\end{scope}
\begin{scope}
\path[clip] (  0.00,  0.00) rectangle (122.86,108.41);
\definecolor{drawColor}{gray}{0.10}

\node[text=drawColor,anchor=base,inner sep=0pt, outer sep=0pt, scale=  0.96] at ( 77.42,  7.75) {$k/m$};
\end{scope}
\begin{scope}
\path[clip] (  0.00,  0.00) rectangle (122.86,108.41);
\definecolor{drawColor}{gray}{0.10}

\node[text=drawColor,rotate= 90.00,anchor=base,inner sep=0pt, outer sep=0pt, scale=  0.85] at ( 12.49, 67.04) {Relative Error};
\end{scope}
\begin{scope}
\path[clip] (  0.00,  0.00) rectangle (122.86,108.41);
\definecolor{drawColor}{RGB}{0,0,0}

\path[draw=drawColor,line width= 0.4pt,line join=round,line cap=round] ( 39.81, 91.35) rectangle ( 43.30, 94.84);
\end{scope}
\begin{scope}
\path[clip] (  0.00,  0.00) rectangle (122.86,108.41);
\definecolor{drawColor}{RGB}{0,0,0}

\path[draw=drawColor,line width= 0.4pt,line join=round,line cap=round] ( 41.56, 85.81) circle (  1.96);
\end{scope}
\begin{scope}
\path[clip] (  0.00,  0.00) rectangle (122.86,108.41);
\definecolor{drawColor}{RGB}{0,0,0}
\definecolor{fillColor}{RGB}{255,0,0}

\path[draw=drawColor,line width= 0.4pt,line join=round,line cap=round,fill=fillColor] ( 41.56, 78.52) circle (  1.96);
\end{scope}
\begin{scope}
\path[clip] (  0.00,  0.00) rectangle (122.86,108.41);
\definecolor{drawColor}{RGB}{0,0,0}

\node[text=drawColor,anchor=base west,inner sep=0pt, outer sep=0pt, scale=  0.69] at ( 46.29, 90.25) {Least-Squares$^*$};
\end{scope}
\begin{scope}
\path[clip] (  0.00,  0.00) rectangle (122.86,108.41);
\definecolor{drawColor}{RGB}{0,0,0}

\node[text=drawColor,anchor=base west,inner sep=0pt, outer sep=0pt, scale=  0.69] at ( 46.29, 82.97) {Least-Squares};
\end{scope}
\begin{scope}
\path[clip] (  0.00,  0.00) rectangle (122.86,108.41);
\definecolor{drawColor}{RGB}{0,0,0}

\node[text=drawColor,anchor=base west,inner sep=0pt, outer sep=0pt, scale=  0.69] at ( 46.29, 75.68) {$\texttt{IRLS}_{0.1}$};
\end{scope}
\end{tikzpicture} \label{fig:EYaleB-err} }
	\subfloat[Qualitative results (randomly selected images)]{ \label{fig:EYaleB-table}
	\begin{tabular}[b]{ccccc}
		\toprule
		\scriptsize{Ground-Truth $\bm{f}_i$} & \scriptsize{Measurements $\tilde{\bm{f}}_i$} & \scriptsize{Least-Squares$^*$} & \scriptsize{Least-Squares} & \scriptsize{$\texttt{IRLS}_{0.1}$} \\
		\midrule 
		\includegraphics[width=0.05\textwidth]{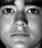}&\includegraphics[width=0.05\textwidth]{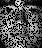} &\includegraphics[width=0.05\textwidth]{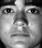}&\includegraphics[width=0.05\textwidth]{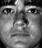}& \includegraphics[width=0.05\textwidth]{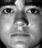} \\
		\includegraphics[width=0.05\textwidth]{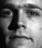}&\includegraphics[width=0.05\textwidth]{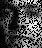} &\includegraphics[width=0.05\textwidth]{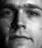}&\includegraphics[width=0.05\textwidth]{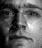}& \includegraphics[width=0.05\textwidth]{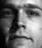} \\
		\includegraphics[width=0.05\textwidth]{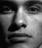}&\includegraphics[width=0.05\textwidth]{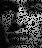} &\includegraphics[width=0.05\textwidth]{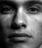}&\includegraphics[width=0.05\textwidth]{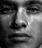}& \includegraphics[width=0.05\textwidth]{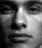} \\
		\bottomrule
	\end{tabular}	
}
\caption{Recover a face image of the Extended Yale B dataset \cite{Georghiades-PAMI2001} that is corrupted by \textit{salt \& pepper} noise \cite{Solomon-2011}. Figure \ref{fig:EYaleB-err}: relative error as a function of the amount of random salt \& pepper noise or sparsity level, averaged over $20$ trials, all faces, and all individuals. Figure \ref{fig:EYaleB-table}: qualitative results. \label{fig:EYaleB} }
\end{figure}
 
In our experiments of Figure \ref{fig:EYaleB}, we corrupt $\bm{f}_i$ by random salt \& pepper noise \cite{Solomon-2011} and obtain the noisy face $\tilde{\bm{f}}_i$ as our measurement; so $\bF_{i} \bx_i^*- \tilde{\bm{f}}_i$ is approximately sparse. We consider three methods to estimate $\bx_i^*$ and thus to find an estimate $\hat{\bm{f}}_i$ of $\bm{f}_i$: (1) Least-Squares that solves \eqref{eq:Lp} with $\bA =\bF_i$, $\by=\tilde{\bm{f}}_i$, and $p=2$, (2) Least-Squares$^*$ that solves \eqref{eq:Lp} with $\bA =\bF_i$, $\by=\bm{f}_i$, and $p=2$, used as a golden baseline, and (3) $\texttt{IRLS}_{0.1}$ which solves \eqref{eq:Lp} with $\bA =\bF_i$ and $\by=\tilde{\bm{f}}_i$. Figure \ref{fig:EYaleB-err} shows that $\texttt{IRLS}_{0.1}$ is more robust to salt \& pepper noise, whose amount is $k/m$, than Least-Squares, and has lower error $\| \hat{\bm{f}}_i-\bm{f}_i\|_2 /\| \bm{f}_i \|_2$ (recall though that Least-Squares is statistically optimal for Gaussian noise). Figure \ref{fig:EYaleB-table} shows some images randomly chosen from the dataset; observe that Least-Squares tends to ``average'' the illumination in all images $\bF_i$, while $\texttt{IRLS}_{0.1}$ delivers faithful restoration.

\subsection{The Choice of $p$ and Performance of $\texttt{IRLS}_p$ Under Noise}\label{subsection:choice-p}
Here we outline two experiments that illustrate two different aspects of $\texttt{IRLS}_p$ empirically. 

In the first experiment, visualized in Figure \ref{fig:LpRR_noiseless}, we compare the decay of the relative errors of the iterates of $\texttt{IRLS}_p$ for different choices of $p$ in the case that the residual is exactly $k$-sparse. We observe that $\texttt{IRLS}_1$ exhibits a linear error decay on the one hand and a superlinear decay for  $\texttt{IRLS}_p$ with $p=0.5$ and $p=0.1$ that accelerates as $p$ decreases towards $0$, confirming the rates predicted by Theorem \ref{theorem:GLC-L1} and Theorem \ref{theorem:LSC-noiseless}. ($p=1,0.5,0.1$ are rather arbitrary choices in our experiments.)

In another experiment, visualized in Figure \ref{fig:LpRR_noisy}, we explore the robustness of $\texttt{IRLS}_p$ if the model assumption of $k$-sparse residuals is not satisfied, but only approximately. In particular, the input vector $\by \in \bbR^m$ of $\texttt{IRLS}_p$ is now such that its restriction on $S^*$ with $|S^*|= k$ has i.i.d. $\cN(0,1)$ entries (corresponding to the sparse corruptions), but its other coordinates are distributed as independent Gaussian variables with mean $\ba_i^\top \bx^*$  and variance $0.01$, where $i \in (S^*)^c$ (corresponding to dense noise); see also the code provided in Appendix \ref{subsection:data-generation}. We observe that also for this case of only approximately $k$-sparse residuals, $\texttt{IRLS}_p$ performs well, with faster and more accurate convergence for $p=0.1$ and $p=0.5$ compared to $p=1$. We note that, strictly speaking, Theorem \ref{theorem:GLC-L1} and Theorem \ref{theorem:LSC-noiseless} do not apply directly to this case of approximately $k$-sparse residuals, but it is not hard to generalize it to this case; see \cite[Theorem A.1]{Kummerle-NeurIPS2021} for a similar result for IRLS applied to basis pursuit.

\begin{figure}
	\centering
	\subfloat[The exact $k$-sparse case]{\input{./figures/experimentLpRR_convergence_noiseless.tex} \label{fig:LpRR_noiseless}  } 
	\subfloat[The approximate $k$-sparse case]{
\begin{tikzpicture}[x=1pt,y=1pt]
\definecolor{fillColor}{RGB}{255,255,255}
\path[use as bounding box,fill=fillColor,fill opacity=0.00] (0,0) rectangle (133.70,112.02);
\begin{scope}
\path[clip] (  0.00,  0.00) rectangle (133.70,112.02);
\definecolor{drawColor}{RGB}{255,255,255}
\definecolor{fillColor}{RGB}{255,255,255}

\path[draw=drawColor,line width= 0.6pt,line join=round,line cap=round,fill=fillColor] (  0.00,  0.00) rectangle (133.70,112.02);
\end{scope}
\begin{scope}
\path[clip] ( 39.51, 29.93) rectangle (128.20,106.52);
\definecolor{fillColor}{gray}{0.92}

\path[fill=fillColor] ( 39.51, 29.93) rectangle (128.20,106.52);
\definecolor{drawColor}{RGB}{0,0,0}

\path[draw=drawColor,line width= 0.6pt,line join=round] ( 43.54,103.04) --
	( 52.50, 89.84) --
	( 61.46, 75.55) --
	( 70.42, 61.68) --
	( 79.37, 52.10) --
	( 88.33, 47.25) --
	( 97.29, 44.83) --
	(106.25, 43.68) --
	(115.21, 43.16) --
	(124.17, 42.93);

\path[draw=drawColor,line width= 0.6pt,line join=round] ( 43.54,103.04) --
	( 52.50, 84.07) --
	( 61.46, 54.99) --
	( 70.42, 44.79) --
	( 79.37, 39.06) --
	( 88.33, 35.69) --
	( 97.29, 34.20) --
	(106.25, 33.66) --
	(115.21, 33.48) --
	(124.17, 33.41);

\path[draw=drawColor,line width= 0.6pt,line join=round] ( 43.54,103.04) --
	( 52.50, 80.50) --
	( 61.46, 45.40) --
	( 70.42, 40.05) --
	( 79.37, 36.62) --
	( 88.33, 34.95) --
	( 97.29, 34.35) --
	(106.25, 34.22) --
	(115.21, 34.24) --
	(124.17, 34.29);
\definecolor{fillColor}{RGB}{255,0,0}

\path[draw=drawColor,line width= 0.4pt,line join=round,line cap=round,fill=fillColor] ( 41.99,101.49) rectangle ( 45.09,104.59);

\path[draw=drawColor,line width= 0.4pt,line join=round,line cap=round,fill=fillColor] ( 50.95, 88.29) rectangle ( 54.05, 91.39);

\path[draw=drawColor,line width= 0.4pt,line join=round,line cap=round,fill=fillColor] ( 59.91, 74.00) rectangle ( 63.01, 77.10);

\path[draw=drawColor,line width= 0.4pt,line join=round,line cap=round,fill=fillColor] ( 68.87, 60.13) rectangle ( 71.96, 63.23);

\path[draw=drawColor,line width= 0.4pt,line join=round,line cap=round,fill=fillColor] ( 77.83, 50.55) rectangle ( 80.92, 53.65);

\path[draw=drawColor,line width= 0.4pt,line join=round,line cap=round,fill=fillColor] ( 86.78, 45.70) rectangle ( 89.88, 48.80);

\path[draw=drawColor,line width= 0.4pt,line join=round,line cap=round,fill=fillColor] ( 95.74, 43.28) rectangle ( 98.84, 46.38);

\path[draw=drawColor,line width= 0.4pt,line join=round,line cap=round,fill=fillColor] (104.70, 42.13) rectangle (107.80, 45.23);

\path[draw=drawColor,line width= 0.4pt,line join=round,line cap=round,fill=fillColor] (113.66, 41.62) rectangle (116.76, 44.71);

\path[draw=drawColor,line width= 0.4pt,line join=round,line cap=round,fill=fillColor] (122.62, 41.38) rectangle (125.72, 44.48);

\path[draw=drawColor,line width= 0.4pt,line join=round,line cap=round,fill=fillColor] ( 43.54,105.76) --
	( 45.89,101.68) --
	( 41.19,101.68) --
	cycle;

\path[draw=drawColor,line width= 0.4pt,line join=round,line cap=round,fill=fillColor] ( 52.50, 86.78) --
	( 54.85, 82.71) --
	( 50.14, 82.71) --
	cycle;

\path[draw=drawColor,line width= 0.4pt,line join=round,line cap=round,fill=fillColor] ( 61.46, 57.71) --
	( 63.81, 53.63) --
	( 59.10, 53.63) --
	cycle;

\path[draw=drawColor,line width= 0.4pt,line join=round,line cap=round,fill=fillColor] ( 70.42, 47.51) --
	( 72.77, 43.44) --
	( 68.06, 43.44) --
	cycle;

\path[draw=drawColor,line width= 0.4pt,line join=round,line cap=round,fill=fillColor] ( 79.37, 41.78) --
	( 81.73, 37.70) --
	( 77.02, 37.70) --
	cycle;

\path[draw=drawColor,line width= 0.4pt,line join=round,line cap=round,fill=fillColor] ( 88.33, 38.41) --
	( 90.69, 34.33) --
	( 85.98, 34.33) --
	cycle;

\path[draw=drawColor,line width= 0.4pt,line join=round,line cap=round,fill=fillColor] ( 97.29, 36.92) --
	( 99.65, 32.84) --
	( 94.94, 32.84) --
	cycle;

\path[draw=drawColor,line width= 0.4pt,line join=round,line cap=round,fill=fillColor] (106.25, 36.38) --
	(108.60, 32.30) --
	(103.90, 32.30) --
	cycle;

\path[draw=drawColor,line width= 0.4pt,line join=round,line cap=round,fill=fillColor] (115.21, 36.19) --
	(117.56, 32.12) --
	(112.86, 32.12) --
	cycle;

\path[draw=drawColor,line width= 0.4pt,line join=round,line cap=round,fill=fillColor] (124.17, 36.13) --
	(126.52, 32.05) --
	(121.81, 32.05) --
	cycle;

\path[draw=drawColor,line width= 0.4pt,line join=round,line cap=round,fill=fillColor] ( 43.54,103.04) circle (  1.75);

\path[draw=drawColor,line width= 0.4pt,line join=round,line cap=round,fill=fillColor] ( 52.50, 80.50) circle (  1.75);

\path[draw=drawColor,line width= 0.4pt,line join=round,line cap=round,fill=fillColor] ( 61.46, 45.40) circle (  1.75);

\path[draw=drawColor,line width= 0.4pt,line join=round,line cap=round,fill=fillColor] ( 70.42, 40.05) circle (  1.75);

\path[draw=drawColor,line width= 0.4pt,line join=round,line cap=round,fill=fillColor] ( 79.37, 36.62) circle (  1.75);

\path[draw=drawColor,line width= 0.4pt,line join=round,line cap=round,fill=fillColor] ( 88.33, 34.95) circle (  1.75);

\path[draw=drawColor,line width= 0.4pt,line join=round,line cap=round,fill=fillColor] ( 97.29, 34.35) circle (  1.75);

\path[draw=drawColor,line width= 0.4pt,line join=round,line cap=round,fill=fillColor] (106.25, 34.22) circle (  1.75);

\path[draw=drawColor,line width= 0.4pt,line join=round,line cap=round,fill=fillColor] (115.21, 34.24) circle (  1.75);

\path[draw=drawColor,line width= 0.4pt,line join=round,line cap=round,fill=fillColor] (124.17, 34.29) circle (  1.75);
\end{scope}
\begin{scope}
\path[clip] (  0.00,  0.00) rectangle (133.70,112.02);
\definecolor{drawColor}{gray}{0.10}

\node[text=drawColor,anchor=base east,inner sep=0pt, outer sep=0pt, scale=  0.73] at ( 34.56, 92.11) {1E-1};

\node[text=drawColor,anchor=base east,inner sep=0pt, outer sep=0pt, scale=  0.73] at ( 34.56, 65.82) {1E-2};

\node[text=drawColor,anchor=base east,inner sep=0pt, outer sep=0pt, scale=  0.73] at ( 34.56, 31.61) {5E-4};
\end{scope}
\begin{scope}
\path[clip] (  0.00,  0.00) rectangle (133.70,112.02);
\definecolor{drawColor}{gray}{0.20}

\path[draw=drawColor,line width= 0.6pt,line join=round] ( 36.76, 95.14) --
	( 39.51, 95.14);

\path[draw=drawColor,line width= 0.6pt,line join=round] ( 36.76, 68.85) --
	( 39.51, 68.85);

\path[draw=drawColor,line width= 0.6pt,line join=round] ( 36.76, 34.64) --
	( 39.51, 34.64);
\end{scope}
\begin{scope}
\path[clip] (  0.00,  0.00) rectangle (133.70,112.02);
\definecolor{drawColor}{gray}{0.20}

\path[draw=drawColor,line width= 0.6pt,line join=round] ( 52.50, 27.18) --
	( 52.50, 29.93);

\path[draw=drawColor,line width= 0.6pt,line join=round] ( 70.42, 27.18) --
	( 70.42, 29.93);

\path[draw=drawColor,line width= 0.6pt,line join=round] ( 88.33, 27.18) --
	( 88.33, 29.93);

\path[draw=drawColor,line width= 0.6pt,line join=round] (106.25, 27.18) --
	(106.25, 29.93);

\path[draw=drawColor,line width= 0.6pt,line join=round] (124.17, 27.18) --
	(124.17, 29.93);
\end{scope}
\begin{scope}
\path[clip] (  0.00,  0.00) rectangle (133.70,112.02);
\definecolor{drawColor}{gray}{0.10}

\node[text=drawColor,anchor=base,inner sep=0pt, outer sep=0pt, scale=  0.73] at ( 52.50, 18.92) {$2$};

\node[text=drawColor,anchor=base,inner sep=0pt, outer sep=0pt, scale=  0.73] at ( 70.42, 18.92) {$4$};

\node[text=drawColor,anchor=base,inner sep=0pt, outer sep=0pt, scale=  0.73] at ( 88.33, 18.92) {$6$};

\node[text=drawColor,anchor=base,inner sep=0pt, outer sep=0pt, scale=  0.73] at (106.25, 18.92) {$8$};

\node[text=drawColor,anchor=base,inner sep=0pt, outer sep=0pt, scale=  0.73] at (124.17, 18.92) {$10$};
\end{scope}
\begin{scope}
\path[clip] (  0.00,  0.00) rectangle (133.70,112.02);
\definecolor{drawColor}{gray}{0.10}

\node[text=drawColor,anchor=base,inner sep=0pt, outer sep=0pt, scale=  0.85] at ( 83.85,  7.47) {Number of Iterations};
\end{scope}
\begin{scope}
\path[clip] (  0.00,  0.00) rectangle (133.70,112.02);
\definecolor{drawColor}{gray}{0.10}

\node[text=drawColor,rotate= 90.00,anchor=base,inner sep=0pt, outer sep=0pt, scale=  0.85] at ( 12.49, 68.23) {Relative Error};
\end{scope}
\begin{scope}
\path[clip] (  0.00,  0.00) rectangle (133.70,112.02);
\definecolor{drawColor}{RGB}{0,0,0}
\definecolor{fillColor}{RGB}{255,0,0}

\path[draw=drawColor,line width= 0.4pt,line join=round,line cap=round,fill=fillColor] ( 87.62, 88.06) rectangle ( 90.72, 91.16);
\end{scope}
\begin{scope}
\path[clip] (  0.00,  0.00) rectangle (133.70,112.02);
\definecolor{drawColor}{RGB}{0,0,0}
\definecolor{fillColor}{RGB}{255,0,0}

\path[draw=drawColor,line width= 0.4pt,line join=round,line cap=round,fill=fillColor] ( 89.17, 85.04) --
	( 91.53, 80.97) --
	( 86.82, 80.97) --
	cycle;
\end{scope}
\begin{scope}
\path[clip] (  0.00,  0.00) rectangle (133.70,112.02);
\definecolor{drawColor}{RGB}{0,0,0}
\definecolor{fillColor}{RGB}{255,0,0}

\path[draw=drawColor,line width= 0.4pt,line join=round,line cap=round,fill=fillColor] ( 89.17, 75.04) circle (  1.75);
\end{scope}
\begin{scope}
\path[clip] (  0.00,  0.00) rectangle (133.70,112.02);
\definecolor{drawColor}{RGB}{0,0,0}

\node[text=drawColor,anchor=base west,inner sep=0pt, outer sep=0pt, scale=  0.69] at ( 93.62, 86.77) {$\texttt{IRLS}_{1}$};
\end{scope}
\begin{scope}
\path[clip] (  0.00,  0.00) rectangle (133.70,112.02);
\definecolor{drawColor}{RGB}{0,0,0}

\node[text=drawColor,anchor=base west,inner sep=0pt, outer sep=0pt, scale=  0.69] at ( 93.62, 79.48) {$\texttt{IRLS}_{0.5}$};
\end{scope}
\begin{scope}
\path[clip] (  0.00,  0.00) rectangle (133.70,112.02);
\definecolor{drawColor}{RGB}{0,0,0}

\node[text=drawColor,anchor=base west,inner sep=0pt, outer sep=0pt, scale=  0.69] at ( 93.62, 72.20) {$\texttt{IRLS}_{0.1}$};
\end{scope}
\end{tikzpicture} \label{fig:LpRR_noisy} } 
	\caption{Empirical convergence rates of $\texttt{IRLS}_p$ for different values of $p$ ($m=1000,n=10,k=200$, averaged over $20$ trials). Figure \ref{fig:LpRR_noiseless}: $\bA \bx^* - \by$ is exactly $k$-sparse. Figure \ref{fig:LpRR_noisy}: $\bA \bx^* - \by$ is approximately $k$-sparse with $1\%$ noise (see Section \ref{subsection:data-generation} for data generation). \label{fig:LpRR_noiseless_noisy}  }
\end{figure}

\section{Discussion and Future Work}
In this work, we established novel global linear and local superlinear rates of $\texttt{IRLS}_p$ for $\ell_p$-regression \eqref{eq:Lp} under the assumption of the stable range space property. Furthermore, we explored several applications for which $\texttt{IRLS}_p$ exhibits state-of-the-art results. 

There are several directions that deserve further investigation. Theoretically, the Gaussian assumption of Proposition \ref{prop:Gaussian-RSP} might be weakened; similar results would hold for sub-Gaussian distributions \cite{LecueMendelson-EMS17,DirksenLecueRauhut-IEEEIT2017}. Note also that, in experiments (Figures \ref{fig:RR1} and \ref{fig:application}), $\texttt{IRLS}_p$ works well even when $m-n\to 2k$. This leaves the question of whether the constant or logarithmic factors of condition \eqref{eq:condition-Gaussian-RSP} are too stringent for guaranteeing the stable RSP to hold; can this be improved? Finally, it is tempting to carry out a global rate analysis of $\texttt{IRLS}_p$ with $p\in(0,1)$ for $\ell_p$-regression; can one prove global (linear) rates under the stable RSP or other assumption?

Algorithmically, designing inexact solvers for the inner (weighted) least-squares problem \eqref{eq:x-update} based on Krylov-subspace methods \cite{Fornasier-COA2016,Gazzola-SIAM-J-MA2020} is expected to further accelerate our current IRLS implementation (our current implementation is attached in Appendix \ref{appendix:impl}). Also, our empirical experiments suggest that $\texttt{IRLS}_p$ (in particular $p\in(0,1)$) is worth being adopted and applied to many other outlier-robust estimation tasks beyond robust regression  \cite{Tsakiris-JMLR2018,Aftab-WCACV2015,Dong-PRL2019,Iwata-ECCV2020,Ding-CVPR2020,Yang-RA-L2020,Chatterjee-TPAMI2017, Sidhartha-3DV2021,Peng-CVPR2022,Lee-CVPR2022,Peng-ECCV2022,Carlone-arXiv2022}.

\begin{ack}
This work was supported by grants NSF 1704458, NSF 1934979, NSF-IIS-1837991, ONR MURI 503405-78051, 
and the NSF-Simons Collaboration on the Mathematical Foundations of Deep Learning (NSF grant 2031985).
\end{ack}


{
\small 

\bibliographystyle{ieeetr}
\bibliography{Liangzu}

\begin{thebibliography}{100}

\bibitem{ElAlaoui-ABLpLaplacian2016}
A.~El~Alaoui, X.~Cheng, A.~Ramdas, M.~J. Wainwright, and M.~I. Jordan,
  ``Asymptotic behavior of $\ell_p$-based {Laplacian} regularization in
  semi-supervised learning,'' in {\em Conference on Learning Theory},
  pp.~879--906, 2016.

\bibitem{Adil-NeurIPS2019}
D.~Adil, R.~Peng, and S.~Sachdeva, ``Fast, provably convergent {IRLS} algorithm
  for p-norm linear regression,'' {\em Advances in Neural Information
  Processing Systems}, vol.~32, 2019.

\bibitem{Slepcev-SIAMAnalysis19}
D.~Slepcev and M.~Thorpe, ``Analysis of p-{Laplacian} regularization in
  semisupervised learning,'' {\em SIAM Journal on Mathematical Analysis},
  vol.~51, no.~3, pp.~2085--2120, 2019.

\bibitem{Bubeck-STOC2018}
S.~Bubeck, M.~B. Cohen, Y.~T. Lee, and Y.~Li, ``An homotopy method for $\ell_p$
  regression provably beyond self-concordance and in input-sparsity time,'' in
  {\em ACM Symposium on Theory of Computing}, pp.~1130--1137, 2018.

\bibitem{Adil-SODA2019}
D.~Adil, R.~Kyng, R.~Peng, and S.~Sachdeva, ``Iterative refinement for
  $\ell_p$-norm regression,'' in {\em ACM-SIAM Symposium on Discrete
  Algorithms}, pp.~1405--1424, 2019.

\bibitem{Jambulapati-arXiv2021}
A.~Jambulapati, Y.~P. Liu, and A.~Sidford, ``Improved iteration complexities
  for overconstrained $p$-norm regression,'' tech. rep., arXiv:2111.01848v2
  [cs.DS], 2021.

\bibitem{Boscovich-1757}
R.~J. Boscovich, ``De litteraria expeditione per pontificiam ditionem, et
  synopsis amplioris operis, ac habentur plura ejus ex exemplaria etiam
  sensorum impessa,'' {\em Bononiensi Scientiarum et Artum Instuto Atque
  Academia Commentarii}, vol.~4, pp.~353--396, 1757.

\bibitem{Plackett-Biometrika1972}
R.~L. Plackett, ``{Studies in the History of Probability and Statistics. XXIX:
  The discovery of the method of least squares},'' {\em Biometrika}, vol.~59,
  pp.~239--251, 08 1972.

\bibitem{Wright-2020}
J.~Wright and Y.~Ma, {\em High-Dimensional Data Analysis with Low-Dimensional
  Models: Principles, Computation, and Applications}.
\newblock Cambridge University Press, 2020.

\bibitem{Huber-1964}
P.~J. Huber, ``Robust estimation of a location parameter,'' {\em The Annals of
  Mathematical Statistics}, vol.~35, no.~1, pp.~73--101, 1964.

\bibitem{Candes-FOCS2005}
E.~Candes, M.~Rudelson, T.~Tao, and R.~Vershynin, ``Error correction via linear
  programming,'' in {\em IEEE Symposium on Foundations of Computer Science},
  pp.~668--681, 2005.

\bibitem{Bhatia-NeurIPS2015}
K.~Bhatia, P.~Jain, and P.~Kar, ``Robust regression via hard thresholding,''
  {\em Advances in Neural Information Processing Systems}, vol.~28, 2015.

\bibitem{Bhatia-NeurIPS2017}
K.~Bhatia, P.~Jain, P.~Kamalaruban, and P.~Kar, ``Consistent robust
  regression,'' {\em Advances in Neural Information Processing Systems},
  vol.~30, 2017.

\bibitem{Yang-JMLR2017}
J.~Yang, Y.-L. Chow, C.~R{\'e}, and M.~W. Mahoney, ``Weighted {SGD} for
  $\ell_p$ regression with randomized preconditioning,'' {\em The Journal of
  Machine Learning Research}, vol.~18, no.~1, pp.~7811--7853, 2017.

\bibitem{Durfee-COLT2018}
D.~Durfee, K.~A. Lai, and S.~Sawlani, ``$\ell_1$ regression using lewis weights
  preconditioning and stochastic gradient descent,'' in {\em Conference On
  Learning Theory}, pp.~1626--1656, 2018.

\bibitem{Suggala-COLT2019}
A.~S. Suggala, K.~Bhatia, P.~Ravikumar, and P.~Jain, ``Adaptive hard
  thresholding for near-optimal consistent robust regression,'' in {\em
  Conference on Learning Theory}, pp.~2892--2897, 2019.

\bibitem{Mukhoty-AISTATS2019}
B.~Mukhoty, G.~Gopakumar, P.~Jain, and P.~Kar, ``Globally-convergent
  iteratively reweighted least squares for robust regression problems,'' in
  {\em International Conference on Artificial Intelligence and Statistics},
  pp.~313--322, 2019.

\bibitem{Pesme-NeurIPS2020}
S.~Pesme and N.~Flammarion, ``Online robust regression via sgd on the $\ell_1$
  loss,'' {\em Advances in Neural Information Processing Systems}, vol.~33,
  pp.~2540--2552, 2020.

\bibitem{Chen-COLT2021}
X.~Chen and M.~Derezinski, ``Query complexity of least absolute deviation
  regression via robust uniform convergence,'' in {\em Conference on Learning
  Theory}, pp.~1144--1179, 2021.

\bibitem{Parulekar-arXiv2021}
A.~Parulekar, A.~Parulekar, and E.~Price, ``$\ell_1$ regression with {Lewis}
  weights subsampling,'' tech. rep., arXiv:2105.09433 [cs.LG], 2021.

\bibitem{Candes-TIT2005}
E.~J. Candes and T.~Tao, ``Decoding by linear programming,'' {\em IEEE
  Transactions on Information Theory}, vol.~51, no.~12, pp.~4203--4215, 2005.

\bibitem{Candes-CPAM2006}
E.~J. Candes, J.~K. Romberg, and T.~Tao, ``Stable signal recovery from
  incomplete and inaccurate measurements,'' {\em Communications on Pure and
  Applied Mathematics}, vol.~59, no.~8, pp.~1207--1223, 2006.

\bibitem{Donoho-CPAM2006}
D.~L. Donoho, ``For most large underdetermined systems of linear equations the
  minimal $\ell_1$-norm solution is also the sparsest solution,'' {\em
  Communications on Pure and Applied Mathematics}, vol.~59, no.~6,
  pp.~797--829, 2006.

\bibitem{Cohen-JAMS2009}
A.~Cohen, W.~Dahmen, and R.~DeVore, ``Compressed sensing and best $k$-term
  approximation,'' {\em Journal of the American Mathematical Society}, vol.~22,
  no.~1, pp.~211--231, 2009.

\bibitem{Blumensath-ACHA2009}
T.~Blumensath and M.~E. Davies, ``Iterative hard thresholding for compressed
  sensing,'' {\em Applied and Computational Harmonic Analysis}, vol.~27, no.~3,
  pp.~265--274, 2009.

\bibitem{Daubechies-CPAM2010}
I.~Daubechies, R.~DeVore, M.~Fornasier, and C.~S. G{\"u}nt{\"u}rk,
  ``Iteratively reweighted least squares minimization for sparse recovery,''
  {\em Communications on Pure and Applied Mathematics}, vol.~63, no.~1,
  pp.~1--38, 2010.

\bibitem{Kummerle-NeurIPS2021}
C.~K{\"u}mmerle, C.~Mayrink~Verdun, and D.~St{\"o}ger, ``Iteratively reweighted
  least squares for basis pursuit with global linear convergence rate,'' {\em
  Advances in Neural Information Processing Systems}, 2021.

\bibitem{Foucart-2013}
S.~Foucart and H.~Rauhut, {\em A Mathematical Introduction to Compressive
  Sensing}.
\newblock Springer New York, 2013.

\bibitem{Chartrand-ICASSP2007}
R.~Chartrand, ``Nonconvex compressed sensing and error correction,'' in {\em
  IEEE International Conference on Acoustics, Speech and Signal Processing},
  vol.~3, pp.~III--889, 2007.

\bibitem{Chartrand-SPL2007}
R.~Chartrand, ``Exact reconstruction of sparse signals via nonconvex
  minimization,'' {\em IEEE Signal Processing Letters}, vol.~14, no.~10,
  pp.~707--710, 2007.

\bibitem{Chartrand-2008}
R.~Chartrand and V.~Staneva, ``Restricted isometry properties and nonconvex
  compressive sensing,'' {\em Inverse Problems}, vol.~24, no.~3, p.~035020,
  2008.

\bibitem{Foucart-ACHA2009}
S.~Foucart and M.-J. Lai, ``Sparsest solutions of underdetermined linear
  systems via $\ell_q$-minimization for $0< q \leq 1$,'' {\em Applied and
  Computational Harmonic Analysis}, vol.~26, no.~3, pp.~395--407, 2009.

\bibitem{Wang-TIT2011}
M.~Wang, W.~Xu, and A.~Tang, ``On the performance of sparse recovery via
  $\ell_p$-minimization ($0\leq p \leq 1$),'' {\em IEEE Transactions on
  Information Theory}, vol.~57, no.~11, pp.~7255--7278, 2011.

\bibitem{Sun-ACHA2012}
Q.~Sun, ``Recovery of sparsest signals via $\ell_q$-minimization,'' {\em
  Applied and Computational Harmonic Analysis}, vol.~32, no.~3, pp.~329--341,
  2012.

\bibitem{Ba-TSP2013}
D.~Ba, B.~Babadi, P.~L. Purdon, and E.~N. Brown, ``Convergence and stability of
  iteratively re-weighted least squares algorithms,'' {\em IEEE Transactions on
  Signal Processing}, vol.~62, no.~1, pp.~183--195, 2013.

\bibitem{Zheng-TIT2017}
L.~Zheng, A.~Maleki, H.~Weng, X.~Wang, and T.~Long, ``Does
  $\ell_p$-minimization outperform $\ell_1$-minimization?,'' {\em IEEE
  Transactions on Information Theory}, vol.~63, no.~11, pp.~6896--6935, 2017.

\bibitem{Marjanovic-TSP2012}
G.~Marjanovic and V.~Solo, ``On $\ell_q$ optimization and matrix completion,''
  {\em IEEE Transactions on Signal Processing}, vol.~60, no.~11,
  pp.~5714--5724, 2012.

\bibitem{Mohan-JMLR2012}
K.~Mohan and M.~Fazel, ``Iterative reweighted algorithms for matrix rank
  minimization,'' {\em The Journal of Machine Learning Research}, vol.~13,
  no.~1, pp.~3441--3473, 2012.

\bibitem{Nie-AAAI2012}
F.~Nie, H.~Huang, and C.~Ding, ``Low-rank matrix recovery via efficient
  {Schatten} $p$-norm minimization,'' in {\em AAAI Conference on Artificial
  Intelligence}, 2012.

\bibitem{Kummerle-JMLR2018}
C.~K{\"u}mmerle and J.~Sigl, ``Harmonic mean iteratively reweighted least
  squares for low-rank matrix recovery,'' {\em The Journal of Machine Learning
  Research}, vol.~19, no.~1, pp.~1815--1863, 2018.

\bibitem{Giampouras-NeurIPS2020}
P.~Giampouras, R.~Vidal, A.~Rontogiannis, and B.~Haeffele, ``A novel
  variational form of the {Schatten}-$p$ quasi-norm,'' {\em Advances in Neural
  Information Processing Systems}, vol.~33, pp.~21453--21463, 2020.

\bibitem{Kummerle-ICML2021}
C.~K{\"u}mmerle and C.~Mayrink~Verdun, ``A scalable second order method for
  ill-conditioned matrix completion from few samples,'' in {\em International
  Conference on Machine Learning}, pp.~5872--5883, 2021.

\bibitem{Lerman-FoCM2015}
G.~Lerman, M.~B. McCoy, J.~A. Tropp, and T.~Zhang, ``Robust computation of
  linear models by convex relaxation,'' {\em Foundations of Computational
  Mathematics}, vol.~15, no.~2, pp.~363--410, 2015.

\bibitem{Tsakiris-JMLR2018}
M.~C. Tsakiris and R.~Vidal, ``Dual principal component pursuit,'' {\em Journal
  of Machine Learning Research}, vol.~19, no.~18, pp.~1--50, 2018.

\bibitem{Zhu-NeurIPS2018}
Z.~Zhu, Y.~Wang, D.~Robinson, D.~Naiman, R.~Vidal, and M.~C. Tsakiris, ``Dual
  principal component pursuit: Improved analysis and efficient algorithms,'' in
  {\em Advances in Neural Information Processing Systems}, 2018.

\bibitem{Ding-ICML2019}
T.~Ding, Z.~Zhu, T.~Ding, Y.~Yang, R.~Vidal, M.~C. Tsakiris, and D.~Robinson,
  ``Noisy dual principal component pursuit,'' in {\em International Conference
  on Machine Learning}, pp.~1617--1625, 2019.

\bibitem{Aftab-WCACV2015}
K.~Aftab and R.~Hartley, ``Convergence of iteratively re-weighted least squares
  to robust {M}-estimators,'' in {\em IEEE Winter Conference on Applications of
  Computer Vision}, pp.~480--487, 2015.

\bibitem{Dong-PRL2019}
W.~Dong, X.-j. Wu, and J.~Kittler, ``Sparse subspace clustering via smoothed
  $\ell_p$ minimization,'' {\em Pattern Recognition Letters}, vol.~125,
  pp.~206--211, 2019.

\bibitem{Iwata-ECCV2020}
D.~Iwata, M.~Waechter, W.-Y. Lin, and Y.~Matsushita, ``An analysis of sketched
  {IRLS} for accelerated sparse residual regression,'' in {\em European
  Conference on Computer Vision}, pp.~609--626, 2020.

\bibitem{Ding-CVPR2020}
T.~Ding, Y.~Yang, Z.~Zhu, D.~P. Robinson, R.~Vidal, L.~Kneip, and M.~C.
  Tsakiris, ``Robust homography estimation via dual principal component
  pursuit,'' in {\em IEEE Conference on Computer Vision and Pattern
  Recognition}, pp.~6080--6089, 2020.

\bibitem{Yang-RA-L2020}
H.~Yang, P.~Antonante, V.~Tzoumas, and L.~Carlone, ``Graduated non-convexity
  for robust spatial perception: From non-minimal solvers to global outlier
  rejection,'' {\em IEEE Robotics and Automation Letters}, vol.~5, no.~2,
  pp.~1127--1134, 2020.

\bibitem{Zhao-SIAM-J-O2012}
Y.-B. Zhao and D.~Li, ``Reweighted $\ell_1$-minimization for sparse solutions
  to underdetermined linear systems,'' {\em SIAM Journal on Optimization},
  vol.~22, no.~3, pp.~1065--1088, 2012.

\bibitem{Beck-OptBook2017}
A.~Beck, {\em First-Order Methods in Optimization}.
\newblock Society for Industrial and Applied Mathematics, 2017.

\bibitem{Nesterov-2018}
Y.~Nesterov, {\em Lectures on Convex Optimization}.
\newblock Springer, 2018.

\bibitem{Chen-MP2019}
Y.~Chen, Y.~Chi, J.~Fan, and C.~Ma, ``Gradient descent with random
  initialization: fast global convergence for nonconvex phase retrieval,'' {\em
  Mathematical Programming}, vol.~176, no.~1, pp.~5--37, 2019.

\bibitem{Tan-AA-J-IMA2019}
Y.~S. Tan and R.~Vershynin, ``Phase retrieval via randomized {K}aczmarz:
  Theoretical guarantees,'' {\em Information and Inference: A Journal of the
  IMA}, vol.~8, no.~1, pp.~97--123, 2019.

\bibitem{Unnikrishnan-TIT18}
J.~Unnikrishnan, S.~Haghighatshoar, and M.~Vetterli, ``Unlabeled sensing with
  random linear measurements,'' {\em IEEE Transactions on Information Theory},
  vol.~64, no.~5, pp.~3237--3253, 2018.

\bibitem{Slawski-JoS19}
M.~Slawski and E.~Ben-David, ``Linear regression with sparsely permuted data,''
  {\em Electronic Journal of Statistics}, vol.~13, no.~1, pp.~1--36, 2019.

\bibitem{Tsakiris-TIT2020}
M.~C. {Tsakiris}, L.~{Peng}, A.~{Conca}, L.~{Kneip}, Y.~{Shi}, and H.~{Choi},
  ``An algebraic-geometric approach for linear regression without
  correspondences,'' {\em IEEE Transactions on Information Theory}, vol.~66,
  no.~8, pp.~5130--5144, 2020.

\bibitem{Peng-SPL2020}
L.~{Peng} and M.~C. {Tsakiris}, ``Linear regression without correspondences via
  concave minimization,'' {\em IEEE Signal Processing Letters}, vol.~27,
  pp.~1580--1584, 2020.

\bibitem{Wright-TPAMI2008}
J.~Wright, A.~Y. Yang, A.~Ganesh, S.~S. Sastry, and Y.~Ma, ``Robust face
  recognition via sparse representation,'' {\em IEEE Transactions on Pattern
  Analysis and Machine Intelligence}, vol.~31, no.~2, pp.~210--227, 2008.

\bibitem{Solomon-2011}
C.~Solomon and T.~Breckon, {\em Fundamentals of Digital Image Processing: A
  practical approach with examples in Matlab}.
\newblock John Wiley \& Sons, 2011.

\bibitem{Yao-NeurIPS2021}
Y.~Yao, L.~Peng, and M.~Tsakiris, ``Unlabeled principal component analysis,''
  {\em Advances in Neural Information Processing Systems}, 2021.

\bibitem{Balan-ACHA2006}
R.~Balan, P.~Casazza, and D.~Edidin, ``On signal reconstruction without
  phase,'' {\em Applied and Computational Harmonic Analysis}, vol.~20, no.~3,
  pp.~345 -- 356, 2006.

\bibitem{Strohmer-JFAA2009}
T.~Strohmer and R.~Vershynin, ``A randomized {K}aczmarz algorithm with
  exponential convergence,'' {\em Journal of Fourier Analysis and
  Applications}, vol.~15, no.~2, pp.~262--278, 2009.

\bibitem{Wei-IP2015}
K.~Wei, ``Solving systems of phaseless equations via {K}aczmarz methods: A
  proof of concept study,'' {\em Inverse Problems}, vol.~31, no.~12, p.~125008,
  2015.

\bibitem{Dhifallah-Allerton2017}
O.~Dhifallah, C.~Thrampoulidis, and Y.~M. Lu, ``Phase retrieval via linear
  programming: Fundamental limits and algorithmic improvements,'' in {\em
  Allerton Conference on Communication, Control, and Computing},
  pp.~1071--1077, 2017.

\bibitem{Chen-NIPS2015}
Y.~Chen and E.~Candes, ``Solving random quadratic systems of equations is
  nearly as easy as solving linear systems,'' {\em Advances in Neural
  Information Processing Systems}, vol.~28, 2015.

\bibitem{Zeng-arXiv2017}
W.-J. Zeng and H.-C. So, ``Coordinate descent algorithms for phase retrieval,''
  tech. rep., arXiv:1706.03474 [cs.IT], 2017.

\bibitem{Applegate-NeurIPS2021}
D.~Applegate, M.~D{\'i}az, O.~Hinder, H.~Lu, M.~Lubin, B.~O'Donoghue, and
  W.~Schudy, ``Practical large-scale linear programming using primal-dual
  hybrid gradient,'' {\em Advances in Neural Information Processing Systems},
  vol.~34, 2021.

\bibitem{Gurobi951}
``Gurobi 9.5.1.'' https://www.gurobi.com/, Date Accessed: May 10, 2022.

\bibitem{Beck-OMS2019}
A.~Beck and N.~Guttmann-Beck, ``{FOM} – a {MATLAB} toolbox of first-order
  methods for solving convex optimization problems,'' {\em Optimization Methods
  and Software}, vol.~34, no.~1, pp.~172--193, 2019.

\bibitem{Georghiades-PAMI2001}
A.~S. Georghiades, P.~N. Belhumeur, and D.~J. Kriegman, ``From few to many:
  Illumination cone models for face recognition under variable lighting and
  pose,'' {\em IEEE Transactions on Pattern Analysis and Machine Intelligence},
  vol.~23, no.~6, pp.~643--660, 2001.

\bibitem{Chen-SIAM-Rev2001}
S.~S. Chen, D.~L. Donoho, and M.~A. Saunders, ``Atomic decomposition by basis
  pursuit,'' {\em SIAM Review}, vol.~43, no.~1, pp.~129--159, 2001.

\bibitem{Lerman-J-IMA2018}
G.~Lerman and T.~Maunu, ``Fast, robust and non-convex subspace recovery,'' {\em
  Information and Inference: A Journal of the IMA}, vol.~7, no.~2,
  pp.~277--336, 2018.

\bibitem{Qu-arXiv2020}
Q.~Qu, Z.~Zhu, X.~Li, M.~C. Tsakiris, J.~Wright, and R.~Vidal, ``Finding the
  sparsest vectors in a subspace: Theory, algorithms, and applications,'' tech.
  rep., arXiv:2001.06970 [cs.LG], 2020.

\bibitem{Chan-SIAM-J-NA1999}
T.~F. Chan and P.~Mulet, ``On the convergence of the lagged diffusivity fixed
  point method in total variation image restoration,'' {\em SIAM Journal on
  Numerical Analysis}, vol.~36, no.~2, pp.~354--367, 1999.

\bibitem{Beck-SIAMOpt2015}
A.~Beck, ``On the convergence of alternating minimization for convex
  programming with applications to iteratively reweighted least squares and
  decomposition schemes,'' {\em SIAM Journal on Optimization}, vol.~25, no.~1,
  pp.~185--209, 2015.

\bibitem{Aravkin-arXiv2019}
A.~Y. Aravkin, J.~V. Burke, and D.~He, ``{IRLS} for sparse recovery revisited:
  Examples of failure and a remedy,'' tech. rep., arXiv:1910.07095 [math.ST],
  2019.

\bibitem{Yang-arXiv2021Lp}
X.~Yang, J.~Wang, and H.~Wang, ``Towards an efficient approach for the
  nonconvex $\ell_p$-ball projection: Algorithm and analysis,'' tech. rep.,
  arXiv:2101.01350v5 [math.OC], 2021.

\bibitem{Tillmann-TIT2013}
A.~M. Tillmann and M.~E. Pfetsch, ``The computational complexity of the
  restricted isometry property, the nullspace property, and related concepts in
  compressed sensing,'' {\em IEEE Transactions on Information Theory}, vol.~60,
  no.~2, pp.~1248--1259, 2013.

\bibitem{Arora-2009}
S.~Arora and B.~Barak, {\em Computational Complexity: A Modern Approach}.
\newblock Cambridge University Press, 2009.

\bibitem{Boyd-book2004}
S.~Boyd, S.~P. Boyd, and L.~Vandenberghe, {\em Convex optimization}.
\newblock Cambridge University Press, 2004.

\bibitem{Allen-JMLR2017}
Z.~Allen-Zhu, ``Katyusha: The first direct acceleration of stochastic gradient
  methods,'' {\em The Journal of Machine Learning Research}, vol.~18, no.~1,
  pp.~8194--8244, 2017.

\bibitem{LecueMendelson-EMS17}
G.~Lecu{\'e} and S.~Mendelson, ``Sparse recovery under weak moment
  assumptions,'' {\em Journal of the European Mathematical Society}, vol.~19,
  no.~3, pp.~881--904, 2017.

\bibitem{DirksenLecueRauhut-IEEEIT2017}
S.~Dirksen, G.~Lecu{\'e}, and H.~Rauhut, ``On the gap between restricted
  isometry properties and sparse recovery conditions,'' {\em IEEE Transactions
  on Information Theory}, vol.~64, no.~8, pp.~5478--5487, 2016.

\bibitem{Tsakiris-ICML2019}
M.~C. Tsakiris and L.~Peng, ``Homomorphic sensing,'' in {\em International
  Conference on Machine Learning}, 2019.

\bibitem{Peng-ICML2021}
L.~Peng, B.~Wang, and M.~Tsakiris, ``Homomorphic sensing: Sparsity and noise,''
  in {\em International Conference on Machine Learning}, pp.~8464--8475, 2021.

\bibitem{Peng-ACHA2021}
L.~Peng and M.~C. Tsakiris, ``Homomorphic sensing of subspace arrangements,''
  {\em Applied and Computational Harmonic Analysis}, vol.~55, pp.~466--485,
  2021.

\bibitem{Grohs-SIAM-Review2020}
P.~Grohs, S.~Koppensteiner, and M.~Rathmair, ``Phase retrieval: Uniqueness and
  stability,'' {\em SIAM Review}, vol.~62, no.~2, pp.~301--350, 2020.

\bibitem{Huang-arXiv2021b}
M.~Huang and Y.~Wang, ``Linear convergence of randomized {K}aczmarz method for
  solving complex-valued phaseless equations,'' tech. rep., arXiv:2109.11811
  [math.NA], 2021.

\bibitem{Chandra-SampTA2019}
R.~Chandra, T.~Goldstein, and C.~Studer, ``{PhasePack}: A phase retrieval
  library,'' in {\em International conference on Sampling Theory and
  Applications}, pp.~1--5, 2019.

\bibitem{Pananjady-TIT18}
A.~Pananjady, M.~J. Wainwright, and T.~A. Courtade, ``Linear regression with
  shuffled data: Statistical and computational limits of permutation
  recovery,'' {\em IEEE Transactions on Information Theory}, vol.~64, no.~5,
  pp.~3286--3300, 2018.

\bibitem{Slawski-JCGS2021}
M.~Slawski, G.~Diao, and E.~Ben-David, ``A pseudo-likelihood approach to linear
  regression with partially shuffled data,'' {\em Journal of Computational and
  Graphical Statistics}, vol.~0, no.~0, pp.~1--31, 2021.

\bibitem{Xie-ICLR2021}
Y.~Xie, Y.~Mao, S.~Zuo, H.~Xu, X.~Ye, T.~Zhao, and H.~Zha, ``A hypergradient
  approach to robust regression without correspondence,'' in {\em International
  Conference on Learning Representations}, 2021.

\bibitem{Li-ICCV2021}
F.~Li, K.~Fujiwara, F.~Okura, and Y.~Matsushita, ``Generalized shuffled linear
  regression,'' in {\em IEEE/CVF International Conference on Computer Vision},
  pp.~6474--6483, 2021.

\bibitem{Hsu-NIPS17}
D.~Hsu, K.~Shi, and X.~Sun, ``Linear regression without correspondence,'' in
  {\em Advances in Neural Information Processing Systems}, 2017.

\bibitem{You-CVPR2016}
C.~You, D.~Robinson, and R.~Vidal, ``Scalable sparse subspace clustering by
  orthogonal matching pursuit,'' in {\em IEEE Conference on Computer Vision and
  Pattern Recognition}, pp.~3918--3927, 2016.

\bibitem{Basri-TPAMI2003}
R.~Basri and D.~W. Jacobs, ``Lambertian reflectance and linear subspaces,''
  {\em IEEE Transactions on Pattern Analysis and Machine Intelligence},
  vol.~25, no.~2, pp.~218--233, 2003.

\bibitem{Fornasier-COA2016}
M.~Fornasier, S.~Peter, H.~Rauhut, and S.~Worm, ``Conjugate gradient
  acceleration of iteratively re-weighted least squares methods,'' {\em
  Computational Optimization and Applications}, vol.~65, no.~1, pp.~205--259,
  2016.

\bibitem{Gazzola-SIAM-J-MA2020}
S.~Gazzola, C.~Meng, and J.~G. Nagy, ``Krylov methods for low-rank
  regularization,'' {\em SIAM Journal on Matrix Analysis and Applications},
  vol.~41, no.~4, pp.~1477--1504, 2020.

\bibitem{Chatterjee-TPAMI2017}
A.~Chatterjee and V.~M. Govindu, ``Robust relative rotation averaging,'' {\em
  IEEE Transactions on Pattern Analysis and Machine Intelligence}, vol.~40,
  no.~4, pp.~958--972, 2017.

\bibitem{Sidhartha-3DV2021}
C.~Sidhartha and V.~M. Govindu, ``It is all in the weights: Robust rotation
  averaging revisited,'' in {\em International Conference on 3D Vision},
  pp.~1134--1143, 2021.

\bibitem{Peng-CVPR2022}
L.~Peng, M.~C. Tsakiris, and R.~Vidal, ``{ARCS}: Accurate rotation and
  correspondences search,'' in {\em IEEE/CVF Conference on Computer Vision and
  Pattern Recognition}, pp.~11153--11163, 2022.

\bibitem{Lee-CVPR2022}
S.~H. Lee and J.~Civera, ``{HARA}: A hierarchical approach for robust rotation
  averaging,'' in {\em IEEE/CVF Conference on Computer Vision and Pattern
  Recognition}, pp.~15777--15786, 2022.

\bibitem{Peng-ECCV2022}
L.~Peng, M.~Fazlyab, and R.~Vidal, ``Semidefinite relaxations of truncated
  least-squares in robust rotation search: Tight or not,'' in {\em European
  Conference on Computer Vision}, pp.~0--0, Springer, 2022.

\bibitem{Carlone-arXiv2022}
L.~Carlone, ``Estimation contracts for outlier-robust geometric perception,''
  tech. rep., arXiv:2208.10521 [stat.ML], 2022.

\bibitem{Chen-MP2012}
X.~Chen, ``Smoothing methods for nonsmooth, nonconvex minimization,'' {\em
  Mathematical Programming}, vol.~134, no.~1, pp.~71--99, 2012.

\bibitem{Nesterov-FoCM2017}
Y.~Nesterov and V.~Spokoiny, ``Random gradient-free minimization of convex
  functions,'' {\em Foundations of Computational Mathematics}, vol.~17, no.~2,
  pp.~527--566, 2017.

\bibitem{Beck-JOTA2015}
A.~Beck and S.~Sabach, ``Weiszfeld's method: Old and new results,'' {\em
  Journal of Optimization Theory and Applications}, vol.~164, no.~1, pp.~1--40,
  2015.

\bibitem{Lange-MM2016}
K.~Lange, {\em {MM Optimization Algorithms}}.
\newblock Philadelphia, PA: Society for Industrial and Applied Mathematics,
  2016.

\bibitem{SunBabuPalomar-IEEESP2017}
Y.~Sun, P.~Babu, and D.~P. Palomar, ``{M}ajorization-{M}inimization algorithms
  in signal processing, communications, and machine learning,'' {\em {IEEE}
  Transactions on Signal Processing}, vol.~65, no.~3, pp.~794--816, 2017.

\bibitem{FazelHindiBoyd-ACC2003}
M.~{Fazel}, H.~{Hindi}, and S.~P. {Boyd}, ``Log-det heuristic for matrix rank
  minimization with applications to {H}ankel and {E}uclidean distance
  matrices,'' in {\em {Proceedings of the American Control Conference}},
  vol.~3, pp.~2156--2162, 2003.

\bibitem{CandesWakinBoyd-JFA2008}
E.~Cand\`{e}s, M.~B. Wakin, and S.~Boyd, ``Enhancing sparsity by reweighted
  $\ell_1$ minimization,'' {\em The Journal of Fourier Analysis and
  Applications}, vol.~14, pp.~877--905, 2008.

\bibitem{OchsDBP-SIAMIMS2015}
P.~Ochs, A.~Dosovitskiy, T.~Brox, and T.~Pock, ``On iteratively reweighted
  algorithms for nonsmooth nonconvex optimization in computer vision,'' {\em
  SIAM Journal on Imaging Sciences}, vol.~8, no.~1, pp.~331--372, 2015.

\bibitem{Gordon-1988}
Y.~Gordon, ``On {M}ilman's inequality and random subspaces which escape through
  a mesh in $\mathbb{R}^n$,'' in {\em Geometric Aspects of Functional
  Analysis}, pp.~84--106, Springer, 1988.

\bibitem{Rudelson-CPAM2008}
M.~Rudelson and R.~Vershynin, ``On sparse reconstruction from {F}ourier and
  {G}aussian measurements,'' {\em Communications on Pure and Applied
  Mathematics}, vol.~61, no.~8, pp.~1025--1045, 2008.

\end{thebibliography}
}

%
%
%
%
%
%
%
%
%
%
%
%

\section*{Checklist}


\begin{enumerate}

\item For all authors...
\begin{enumerate}
  \item Do the main claims made in the abstract and introduction accurately reflect the paper's contributions and scope?
    \answerYes{Furthermore, we made a lot of efforts in comparing the theoretical results of prior works and position our contribution in the literature.}
  \item Did you describe the limitations of your work?
    \answerYes{See for example at Lines 268-274 on real phase retrieval. We also discussed several directions for improvements or future research at the end of the paper.}
  \item Did you discuss any potential negative societal impacts of your work?
    \answerNo{}
  \item Have you read the ethics review guidelines and ensured that your paper conforms to them?
    \answerYes{}
\end{enumerate}

\item If you are including theoretical results...
\begin{enumerate}
  \item Did you state the full set of assumptions of all theoretical results?
    \answerYes{In particular, we used an entire subsection (Section \ref{subsection:RSP}) to justify the major assumption, the stable range space property.}
        \item Did you include complete proofs of all theoretical results?
    \answerYes{We omitted the proofs of Proposition \ref{prop:LpRR=LpBP} and \ref{prop:RSP=exact-recovery} as they follow directly from existing results; see Line 83 and Line 155. All other theoretical results are proved in the appendix.}
\end{enumerate}

\item If you ran experiments...
\begin{enumerate}
  \item Did you include the code, data, and instructions needed to reproduce the main experimental results (either in the supplemental material or as a URL)?
    \answerYes{We attached a copy of our code in the supplemental material.}
  \item Did you specify all the training details (e.g., data splits, hyperparameters, how they were chosen)?
    \answerYes{In the appendix, we use an entire section to describe our experimental setup. We did not do that in the main paper in interest of space.}
        \item Did you report error bars (e.g., with respect to the random seed after running experiments multiple times)?
    \answerNo{We did not report error bars for two reasons. The main reason is that some of our figures are quite dense, and adding error bars on them makes it harder to read. A slightly subjective reason is that we believe our current figures are informative enough, correctly deliver the message, and are fair for other methods. }
        \item Did you include the total amount of compute and the type of resources used (e.g., type of GPUs, internal cluster, or cloud provider)?
    \answerNA{}
\end{enumerate}

\item If you are using existing assets (e.g., code, data, models) or curating/releasing new assets...
\begin{enumerate}
  \item If your work uses existing assets, did you cite the creators?
    \answerYes{}
  \item Did you mention the license of the assets?
    \answerNo{In the main paper we used the \hyperlink{http://vision.ucsd.edu/~leekc/ExtYaleDatabase/ExtYaleB.html}{Extended Yale B face dataset}, which is free for research purposes.}
  \item Did you include any new assets either in the supplemental material or as a URL?
    \answerNo{}
  \item Did you discuss whether and how consent was obtained from people whose data you're using/curating?
    \answerYes{}
  \item Did you discuss whether the data you are using/curating contains personally identifiable information or offensive content?
    \answerNo{}
\end{enumerate}

\item If you used crowdsourcing or conducted research with human subjects...
\begin{enumerate}
  \item Did you include the full text of instructions given to participants and screenshots, if applicable?
    \answerNA{}
  \item Did you describe any potential participant risks, with links to Institutional Review Board (IRB) approvals, if applicable?
    \answerNA{}
  \item Did you include the estimated hourly wage paid to participants and the total amount spent on participant compensation?
    \answerNA{}
\end{enumerate}

\end{enumerate}


\newpage

\appendix

\section{IRLS as Smoothing Method Minimizing Quadratic Models}
In this section, we clarify the precise relationship between the steps of $\texttt{IRLS}_p$ as described in Algorithm \ref{algo:IRLS} and the $\ell_p$-objective on the residual, cf. \eqref{eq:Lp}.

\subsection{$\texttt{IRLS}_p$ as Lp-Regression for $0 < p \leq 1$} \label{sec:IRLSp:MM}
In Section \ref{subsection:IRLS-basic}, we shed light on the different rules for choosing the \emph{smoothing parameter} $\epsilon^{(t+1)}$ in Algorithm \ref{algo:IRLS}. In fact, this terminology is used as there is an intimate relationship between the least squares step \eqref{eq:x-update} of $\texttt{IRLS}_p$ and a \emph{smoothed} $\ell_p$-objective $H_{\epsilon}(\cdot)$ which is, for given $\epsilon > 0$, defined as
\begin{equation} \label{eq:Heps:def}
 		H_{\epsilon}(\br) = \sum_{i=1}^m h_{\epsilon}(r_i)
\end{equation}
where $h_{\epsilon}: \bbR \to \bbR$ is a symmetric smoothed one-dimensional function such that
\[
h_{\epsilon}(r) = \begin{cases}
		\frac{1}{p} |r|^p & |r| > \epsilon \\ 
		\frac{1}{2}  \frac{r^2}{\epsilon^{2-p}} + \Big(\frac{1}{p}- \frac{1}{2}  \Big) \epsilon^p & |r| \leq \epsilon
	\end{cases},
\]
which makes $H_{\epsilon}$ a function akin to a scaled $\ell_p$-quasinorm, but which is quadratic around $0$. This definition corresponds to a scaled Huber loss in the case of $p=1$ \cite{Huber-1964}, but extends to the non-convex case if $0 < p < 1$. 

By considering the derivative of $\epsilon \to h_{\epsilon}(r)$ for fixed $r$, it is easy to see that $h_{\epsilon}(r)$, and therefore also $H_{\epsilon}(\br)$, is non-decreasing in $\epsilon$, which implies that
\begin{equation} \label{eq:res:majorization}
\frac{1}{p}  \norm{\bA \bx - \by}{p}^p  = \lim_{\epsilon' \to 0} H_{\epsilon'}(\bA \bx- \by) \leq H_{\epsilon}(\bA \bx- \by) 
\end{equation}
for any $\epsilon > 0$. We note that $H_{\epsilon}(\cdot)$ is continuously differentiable, so that the function $\bbR^{m} \times \bbR_{>0} \to \bbR, (\br , \epsilon) \mapsto H_{\epsilon}(\br) $ can be considered as a \emph{smoothing function} as defined in the non-smooth optimization literature \cite{Chen-MP2012,Nesterov-FoCM2017}. 

In this context, one interpretation of $\texttt{IRLS}_p$ (and of other IRLS methods such as \cite{Beck-JOTA2015,Daubechies-CPAM2010,Mukhoty-AISTATS2019,Kummerle-NeurIPS2021}, with possibly different smoothing functions) is to consider it as a \emph{smoothing method for $\ell_p$-regression with quadratic majorizing models}. Indeed, defining the function $Q_{\epsilon}: \bbR^{m} \times \bbR^{m}  \to \bbR$,
\[
Q_{\epsilon}(\bv, \br) = \sum_{i=1}^{m} q_{\epsilon}(v_i, r_i), \text{ where } q_{\epsilon}(v, r) = h_{\epsilon}(r) + \frac{1}{2}\cdot \frac{v^2 - r^2 }{\max\{ |r|, \epsilon \}^{2-p} },
\]
which is quadratic in $\br$, we observe that the least squares step \eqref{eq:x-update} of $\texttt{IRLS}_p$ satisfies
\begin{align}\label{eq:x-update-Q}
    \bx^{(t+1)} =  \argmin_{\bx\in\bbR^n} Q_{\epsilon} (\bA \bx-\by, \br^{(t)})
\end{align}
where $\br^{(t)} = \bA\bx^{(t)} -\by$. Since $Q_{\epsilon} (\cdot, \br)$ locally coincides with the smoothed function $H_{\epsilon}(\cdot)$ so that
\[
 H_{\epsilon}(\br) = Q_{\epsilon} (\br ,\br) 
\]
for each $\br \in \bbR^m$, and furthermore, since it is not hard to show that $Q_{\epsilon} (\cdot, \br)$ majorizes $H_{\epsilon}(\cdot)$ in the sense that
\begin{equation} \label{eq:majorization:Heps}
H_{\epsilon}(\bv) \leq Q_{\epsilon} (\bv, \br)
\end{equation}
for any $\bv, \br \in \bbR^{m}$ and $\epsilon >0$, it is possible to establish the monotonous decrease of the smoothed $\ell_p$-objective at each iteration of $\texttt{IRLS}_p$ without too much additional work: For each $t = 1,2,\ldots$ we obtain that
	\begin{align}\label{eq:chain}
		H_{\epsilon^{(t+1)}}\big(\br^{(t+1)}\big) \leq H_{\epsilon^{(t)}}\big(\br^{(t+1)}\big) \leq Q_{\epsilon^{(t)}}\big(\br^{(t+1)}, \br^{(t)}\big)\leq Q_{\epsilon^{(t)}}\big(\br^{(t)}, \br^{(t)}\big) = H_{\epsilon^{(t)}}\big(\br^{(t)}\big),
	\end{align}
where the first inequality holds as long as $\epsilon^{(t+1)} \leq \epsilon^{(t)}$, which is satisfied for the smoothing parameter update rules described in Section \ref{subsection:IRLS-basic}.

The monotonicity outlined in \eqref{eq:chain} enables the following interpretation of the iterates $\{\bx^{(t)}\}_{t\geq 1}$ of $\texttt{IRLS}_p$: Recalling the definition $\br^{(t)} = \bA\bx^{(t)} -\by$, it follows that without any assumption---in particular, without assuming any range space property on the feature matrix $\bA$---each accumulation point of $\{\br^{(t)}\}_{t \geq 1}$ is a first-order stationary point of the $\overline{\epsilon}$-smoothed $\ell_p$-objective $H_{\overline{\epsilon}}(\cdot)$ of \eqref{eq:Heps:def} if $\overline{\epsilon}$ is defined as $\overline{\epsilon} = \lim_{t\to \infty} \epsilon^{(t)}$ (see, e.g., \cite[Theorem 5.3]{Daubechies-CPAM2010} for a similar result).

We conclude this section by noting that $\texttt{IRLS}_p$ can be also interpreted within the framework of \emph{majorization-minimization (MM)} \cite{Lange-MM2016,SunBabuPalomar-IEEESP2017} algorithms---albeit, with the crucial difference that unlike for MM methods, the smoothing updates of $\texttt{IRLS}_p$ as described in Algorithm \ref{algo:IRLS} change the underlying objective $H_{\epsilon^{(t)}}(\cdot)$ every time when $\epsilon^{(t)}$ changes, i.e., in many cases at each iteration.

\subsection{$\texttt{IRLS}_p$ as Sum-of-Logarithm Minimization for $p=0$} \label{sec:IRLS0:MM}
The considerations of Section \ref{sec:IRLSp:MM} are tailored for $\texttt{IRLS}_p$ with $0 < p \leq 1$. However, it is valid to use $\texttt{IRLS}_p$ with $p=0$ by setting the weights $w_i^{(t+1)}$ accordingly in \eqref{eq:w-update}. In this case, it is still possible to interpret $\texttt{IRLS}_0$ as a smoothing method with respect to an underlying smoothed surrogate objective. 

In particular, we then define $H_{\epsilon}:\bbR^{m} \to \bbR$ as in \eqref{eq:Heps:def} with
\[
	h_{\epsilon}(r) = 
		\begin{cases}
 	\log(|r|), & \text{ if } |r| > \epsilon, \\
 	\frac{1}{2} \frac{r^2}{\epsilon^2} + \log(\epsilon) - \frac{1}{2}, & \text{ if } |r| \leq \epsilon,
 \end{cases}
\]
i.e., a smoothed \emph{sum-of-logarithm} objective \cite{FazelHindiBoyd-ACC2003,CandesWakinBoyd-JFA2008,OchsDBP-SIAMIMS2015}. As the inequalities \eqref{eq:majorization:Heps} and \eqref{eq:chain} still hold in this case, we conclude that $\texttt{IRLS}_0$ can be interpreted as a smoothing method for sum-of-logarithm minimization on the residuals $r = \bA\bx -\by$. We note that while there is a close relationship between their minimizers, a pointwise majorization of an $\ell_0$-objective $\norm{\bA\bx -\by}{0}$ by the smoothed objective $H_{\epsilon}(\cdot)$ as in \eqref{eq:res:majorization} is \emph{not} possible for $p=0$. 

Another minor difference to the case $0 < p \leq 1$ is that for $p=0$, the smoothed objective $H_{\epsilon^{(t)}}(\br^{(t)})$ of $\texttt{IRLS}_0$ residuals $\br^{(t)}$ might converge to $-\infty$ as $\epsilon \to 0$; however, this is rather a technicality than a deterrence for the numerical performance of the algorithm, which remains excellent. 

On the other hand, the local convergence analysis put forward in Section \ref{subsection:LSC-noiseless} and Theorem \ref{theorem:LSC-noiseless} still applies: In particular, we obtain that $\texttt{IRLS}_0$ as described in Algorithm \ref{algo:IRLS} exhibits \emph{local quadratic convergence} under appropriate assumptions on the feature matrix, cf. Theorem \ref{theorem:LSC-noiseless}.

\section{Proofs of Main Results}

\subsection{Global Linear Convergence for L1-Regression}\label{subsection:appendix-GLC}


\begin{proof}[Proof of Theorem \ref{theorem:GLC-L1}]
	Define $\br^*:= \bA \bx^* - \by$, $\br^{(t)}:= \bA \bx^{(t)} - \by$, and $\bd^{(t)}:=\br^{(t)} - \br^*=\bA\bx^{(t)} - \bA \bx^*$. Recalling \eqref{eq:chain}, for any $s>0$ we have the following chain of inequalities:
	\begin{align} \label{eq:chain2}
		H_{\epsilon^{(t+1)}}\big(\br^{(t+1)}\big) \leq H_{\epsilon^{(t)}}\big(\br^{(t+1)}\big) \leq Q_{\epsilon^{(t)}}\big(\br^{(t+1)}, \br^{(t)}\big)\leq Q_{\epsilon^{(t)}}\big(\br^{(t)} - s\bd^{(t)}, \br^{(t)}\big)
	\end{align}
	In \eqref{eq:chain2}, the first inequality holds because $\epsilon^{(t+1)}\leq \epsilon^{(t)}$, and $\epsilon \mapsto H_{\epsilon}(\cdot)$ is a non-decreasing function in $\epsilon$, the second inequality holds because $H_{\epsilon}(\cdot)$ is majorized by $Q_{\epsilon}(\cdot, \times)$, see also \eqref{eq:chain}. Furthermore, the third inequality holds due to the optimality of $\br^{(t+1)}$ or $\bx^{(t+1)}$ (cf. \eqref{eq:x-update} and \eqref{eq:x-update-Q}). From \eqref{eq:chain2} we obtain
	\begin{align}
		H_{\epsilon^{(t+1)}}\big(\br^{(t+1)}\big)  - H_{\epsilon^{(t)}}\big(\br^{(t)}\big) &\leq  Q_{\epsilon^{(t)}}\big(\br^{(t)} - s\bd^{(t)}, \br^{(t)}\big) - H_{\epsilon^{(t)}}\big(\br^{(t)}\big) \\
		&= \frac{1}{2} \sum_{i=1}^{m}  \frac{\big(r^{(t)}_i-sd^{(t)}_i\big)^2 - \big(r_i^{(t)}\big)^2 }{\max\{ \big|r_i^{(t)}\big|, \epsilon^{(t)} \} } \\
		&=s \sum_{i=1}^{m}  \frac{ -r^{(t)}_i \cdot d^{(t)}_i }{\max\big\{ \big|r_i^{(t)}\big|, \epsilon^{(t)} \big\} } + \frac{s^2}{2} \sum_{i=1}^{m}  \frac{\big(d_i^{(t)}\big)^2 }{\max\big\{ \big|r_i^{(t)}\big|, \epsilon^{(t)} \big\} }. \label{eq:t+t-squared}
	\end{align}
	We next derive upper bounds for the two sums of \eqref{eq:t+t-squared} respectively. Denote by $S^*$ the support of $\br^*$, and define $R:=\{ i: \big|r_i^{(t)}\big|>\epsilon\}\subset \{1,\dots,m\}$. Then the first sum of \eqref{eq:t+t-squared} is
	\begin{align}
		\sum_{i=1}^{m}  \frac{ -r^{(t)}_i \cdot d^{(t)}_i }{\max\big\{ \big|r_i^{(t)}\big|, \epsilon^{(t)} \big\} } &= \sum_{i\in S^*} \frac{-r^{(t)}_i}{\max\big\{ \big|r_i^{(t)}\big|, \epsilon^{(t)} \big\} } \cdot d^{(t)}_i + \sum_{i\in (S^*)^c}  \frac{-r^{(t)}_i \cdot r^{(t)}_i}{\max\big\{ \big|r_i^{(t)}\big|, \epsilon^{(t)} \big\}} \\
		&\leq \sum_{i\in S^*} \big|d^{(t)}_i \big| - \sum_{i\in (S^*)^c\cap R}  \frac{r^{(t)}_i \cdot r^{(t)}_i}{ \big|r_i^{(t)}\big|} - \sum_{i\in (S^*)^c\cap R^c}  \frac{r^{(t)}_i \cdot r^{(t)}_i}{\epsilon^{(t)}} \\
		&\leq \big(\eta - 1 \big)\sum_{i\in (S^*)^c} \big|d^{(t)}_i \big| + \sum_{i\in (S^*)^c\cap R^c}  \big| r^{(t)}_i \big| - \sum_{i\in (S^*)^c\cap R^c}  \frac{r^{(t)}_i \cdot r^{(t)}_i}{\epsilon^{(t)}}.
	\end{align}
	In the last inequality, we applied $(k,\eta)$-RSP as our assumption, and used the fact $d_i^{(t)}=r_i^{(t)}$ for $i\in(S^*)^c$. Since $\big| r^{(t)}_i \big|\leq r^{(t)}_i \cdot r^{(t)}_i/\epsilon^{(t)} + \epsilon^{(t)}/4$, from the above we now arrive at
	\begin{align}
		\sum_{i=1}^{m}  \frac{ -r^{(t)}_i \cdot d^{(t)}_i }{\max\big\{ \big|r_i^{(t)}\big|, \epsilon^{(t)} \big\} } &\leq  \big(\eta - 1 \big)\sum_{i\in (S^*)^c} \big|d^{(t)}_i \big| + \big|(S^*)^c\cap R^c\big| \cdot \frac{\epsilon^{(t)}}{4} \\
		&\leq  \big(\eta - 1 \big)\sum_{i\in (S^*)^c} \big|d^{(t)}_i \big| + \frac{m\cdot \epsilon^{(t)}}{4} \\
		&\leq \big(\eta - 1 \big)\sum_{i\in (S^*)^c} \big|d^{(t)}_i \big| + \frac{1}{4} \sum_{i\in (S^*)^c} \big|d^{(t)}_i \big| = \big(\eta - \frac{3}{4} \big)\sum_{i\in (S^*)^c} \big|d^{(t)}_i \big|.
	\end{align}
	In the last inequality we used the definitions of $d^{(t)}_i$ and $\epsilon^{(t)}$. Next, the second summation of \eqref{eq:t+t-squared} is
	\begin{align}
		\sum_{i=1}^{m}  \frac{\big(d_i^{(t)}\big)^2 }{\max\big\{ \big|r_i^{(t)}\big|, \epsilon^{(t)} \big\} } &\leq \max_i |d^{(t)}_i|\cdot  \frac{\sum_{i=1}^m |d^{(t)}_i|}{\epsilon^{(t)}} \\
		&\leq \eta(\eta+1) \sum_{i\in (S^*)^c} |d^{(t)}_i| \cdot  \frac{(\eta+1) \sum_{i\in (S^*)^c} |d^{(t)}_i|}{\epsilon^{(t)}} \\
		&\leq \frac{49}{16}\eta \frac{\Big(\sum_{i\in (S^*)^c} |d^{(t)}_i|\Big)^2}{\epsilon^{(t)}}. \label{eq:ub2}
	\end{align}
	In the above we used $(k,\eta)$-stable RSP multiple times. With the two upper bounds on the sums of \eqref{eq:t+t-squared}, and with $s':=s\big(\sum_{i\in (S^*)^c} |d^{(t)}_i|\big)$, we now obtain
	\begin{align}
		H_{\epsilon^{(t+1)}}\big(\br^{(t+1)}\big)  - H_{\epsilon^{(t)}}\big(\br^{(t)}\big) \leq s'(\eta - \frac{3}{4}) + \frac{(s')^2}{2}\cdot  \frac{49}{16} \cdot \frac{\eta}{\epsilon^{(t)}}.
	\end{align}
	Since $s$ is an arbitrary  non-negative integer, let $s$ be such that $s'=(12-16\eta)\epsilon^{(t)}/(49\eta)$. Then
	\begin{align}
		H_{\epsilon^{(t+1)}}\big(\br^{(t+1)}\big)  - H_{\epsilon^{(t)}}\big(\br^{(t)}\big) \leq - \frac{(3-4\eta)^2}{98\eta}\cdot \epsilon^{(t)}.
	\end{align}
	The definition of $\epsilon^{(t)}$ \eqref{eq:update-Kummerle} implies there exists some $t'\leq t$ such that $\epsilon^{(t)} = \sigma^{(t')}/m$. Substituting this value of $s'$ into the above inequality gives us the bound
	\begin{align}
		H_{\epsilon^{(t+1)}}\big(\br^{(t+1)}\big) - \norm{\br^*}{1}  &\leq H_{\epsilon^{(t)}}\big(\br^{(t)}\big) - \norm{\br^*}{1} - \frac{(3-4\eta)^2}{98\eta}\cdot \frac{\sigma^{(t')}}{m} \\
		&\leq H_{\epsilon^{(t)}}\big(\br^{(t)}\big) - \norm{\br^*}{1} - \frac{(3-4\eta)^2}{98\eta}\cdot \frac{H_{\epsilon^{(t')}}\big(\br^{(t')}\big) - \norm{\br^*}{1}}{3m} \\
		&\leq H_{\epsilon^{(t)}}\big(\br^{(t)}\big) - \norm{\br^*}{1} - \frac{(3-4\eta)^2}{98\eta}\cdot \frac{H_{\epsilon^{(t)}}\big(\br^{(t)}\big) - \norm{\br^*}{1}}{3m} \\
		&= \Big(1- \frac{(3-4\eta)^2}{294\eta m} \Big) \cdot \Big( H_{\epsilon^{(t)}}\big(\br^{(t)}\big) - \norm{\br^*}{1} \Big).
	\end{align}
	The last two inequalities above are due  to Lemma \ref{lemma:1} and the monotonicity of $H_{\epsilon^{(t)}}(\br^{(t)})$ in $t$. Finally, Note that $H_{\epsilon^{(t+1)}}\big(\br^{(t+1)}\big)$ majorizes $\norm{\br^{(t+1)}}{1}$ and Lemma \ref{lemma:1} implies
	\begin{align}
		H_{\epsilon^{(1)}}\big(\br^{(1)}\big) - \norm{\br^*}{1} \leq 3\sigma^{(1)}\leq 3 \norm{\br^{(1)}}{1}.
	\end{align}
	We have thus obtained
	\begin{align}
	 \norm{\br^{(t+1)}}{1} - \norm{\br^*}{1}\leq	H_{\epsilon^{(t+1)}}\big(\br^{(t+1)}\big) - \norm{\br^*}{1} \leq  \Big(1- \frac{(3-4\eta)^2}{294\eta m} \Big)^t \cdot 3 \norm{\br^{(1)}}{1}.
	\end{align}
	and the proof is now complete.
\end{proof}

\subsection{Local Superlinear Convergence for Lp-Regression}
\begin{proof}[Proof of Theorem \ref{theorem:LSC-noiseless}]
    We first show that the right inequality of \eqref{eq:Lp-rates}, $\mu\cdot \norm{\bA \bx^{(t)} - \bA \bx^* }{1}^{1-p} < 1$, holds true. With \eqref{eq:Lp-convergence-radius} and the definition of $\mu$, this inequality holds whenever we have
    \begin{align} 
        &\ \  2\eta (\eta+1)  (1-c)^{p-2} \cdot \min_{i\in S^*} |\trsp{\ba_i}\bx^* - y_i|^{p-1} \cdot \Big( c\cdot \min_{i\in S^*} |\trsp{\ba_i}\bx^* - y_i| \Big)^{1-p} < 1, \\ 
        \Leftrightarrow &\ \  2\eta (\eta+1)  (1-c)^{p-2} c^{1-p}<1,
    \end{align}
    which is true by the definition of $c$. This finishes proving the right inequality of \eqref{eq:Lp-rates}. Next, one easily verifies \eqref{eq:Lp-rates} implies \eqref{eq:Lp-rates2}. Hence, it remains to prove the left inequality of \eqref{eq:Lp-rates}.
	
	Recall $\br^*:= \bA \bx^* - \by$. Define $\br^{(t)}:= \bA \bx^{(t)} - \by$, and $\bd^{(t)}:=\br^{(t)} - \br^*=\bA\bx^{(t)} - \bA \bx^*$. And recall that $S^*$ denotes the support of the ground-truth residual $\br^*$. We have
	\begin{align}
		\sum_{i=1}^{m} \big(r^*_i+d^{(t+1)}_i\big)d^{(t+1)}_i w^{(t)}_i = 0,
	\end{align}
	a known property for the global minimizer $\bx^{(t+1)}$ of the weighted least-squares problem \eqref{eq:x-update}. So
	\begin{align}
		\sum_{i=1}^m\big(d^{(t+1)}_i\big)^2w^{(t)}_i &= -\sum_{i\in S^*} r^*_i d^{(t+1)}_i w^{(t)}_i \\
		&\leq \sum_{i\in S^*} |r^*_i| \cdot  \big|d^{(t+1)}_i\big| \max \big\{ |r^{(t)}_i|, \epsilon^{(t)} \big\}^{p-2} \\
		&\leq \sum_{i\in S^*} |r^*_i| \cdot \big|d^{(t+1)}_i\big| \cdot  |r^{(t)}_i|^{p-2} \\
		&= \sum_{i\in S^*} |r^*_i| \cdot \big|d^{(t+1)}_i\big| \cdot  |d^{(t)}_i + r^*_i|^{p-2} \\
		&\leq \sum_{i\in S^*} |r^*_i| \cdot \big|d^{(t+1)}_i\big| \cdot  |r^*_i|^{p-2} \cdot (1-c)^{p-2}.
	\end{align}
	In the last step, we used the fact $|d^{(t)}_i|\leq c\cdot |r^*_i|$ for all $i \in S^*$ \eqref{eq:Lp-convergence-radius}, which means $|d^{(t)}_i + r^*_i|\geq (1-c) |r^*_i|$. Thus, with the definition $C:=(1-c)^{p-2} \cdot \min_{i\in S^*} |\trsp{\ba_i}\bx^* - y_i|^{p-1}$, we get
	\begin{equation}
		\begin{split}
			\sum_{i=1}^m\big(d^{(t+1)}_i\big)^2w^{(t)}_i &\leq \sum_{i\in S^*}  \big|d^{(t+1)}_i\big| \cdot  |r^*_i|^{p-1} \cdot (1-c)^{p-2} \\
			&\leq C  \sum_{i\in S^*} \big|d^{(t+1)}_i\big|.
		\end{split}
	\end{equation}
	From the above and the $(k,\eta)$-stable RSP, we can now arrive at
	\begin{align}
		\sum_{i\in (S^*)^c}\big(d^{(t+1)}_i\big)^2w^{(t)}_i \leq \sum_{i=1}^m\big(d^{(t+1)}_i\big)^2w^{(t)}_i \leq C  \sum_{i\in S^*} \big|d^{(t+1)}_i\big| \leq \eta\cdot C  \sum_{i\in (S^*)^c} \big|d^{(t+1)}_i\big|.
	\end{align}
	Next, from the above and the Cauchy–Schwarz inequality, we get that
	\begin{equation}
		\begin{split}
			\sum_{i\in (S^*)^c} \big|d^{(t+1)}_i\big| &= \sum_{i\in (S^*)^c} \big|d^{(t+1)}_i\big| \cdot \big(w^{(t)}_i\big)^{\frac{1}{2}} \cdot \big(w^{(t)}_i\big)^{-\frac{1}{2}} \\
			&\leq \sqrt{ \sum_{i\in (S^*)^c}\big(d^{(t+1)}_i\big)^2w^{(t)}_i } \cdot \sqrt{\sum_{i\in (S^*)^c} \big(w^{(t)}_i\big)^{-1} } \\
			&\leq \sqrt{\eta\cdot C  \sum_{i\in (S^*)^c} \big|d^{(t+1)}_i\big|}  \cdot \sqrt{\sum_{i\in (S^*)^c} \big(w^{(t)}_i\big)^{-1} }.
		\end{split}
	\end{equation}
	We next suppose $d^{(t+1)}_i=r^{(t+1)}_{i}-r^*_{i}\neq 0$ for some $i\in (S^*)^c$, for otherwise $\br^{(t+1)}$ is $k$-sparse and thus $\bx^{(t+1)}=\bx^*$ by the uniqueness of the $k$-sparse residual. Then the above inequality implies
	\begin{equation}
		\begin{split} 
			\sum_{i\in (S^*)^c} \big|d^{(t+1)}_i\big| &\leq \eta\cdot C \sum_{i\in (S^*)^c} \big(w^{(t)}_i\big)^{-1} \\
			&= \eta\cdot C \cdot \sum_{i\in (S^*)^c}  \max\big\{ |r^{(t)}_i|, \epsilon^{(t)} \big\}^{2-p} \\
			&\leq \eta\cdot C \cdot \bigg(  m  \big(\epsilon^{(t)}\big)^{2-p} + \sum_{i\in (S^*)^c} \big|r^{(t)}_i \big|^{2-p} \bigg) \\
			&\leq \eta\cdot C \cdot \bigg(  m  \cdot \frac{\big(\sigma^{(t)}\big)^{2-p}}{m^{2-p}} + \sum_{i\in (S^*)^c} \big|d^{(t)}_i \big|^{2-p} \bigg) \\
			&\leq  \eta\cdot C \cdot \bigg(  \frac{1}{m^{1-p}}  \cdot \Big(  \sum_{i\in (S^*)^c} \big|d^{(t)}_i \big| \Big)^{2-p} + \Big( \sum_{i\in (S^*)^c} \big|d^{(t)}_i \big| \Big)^{2-p} \bigg) \\
			&\leq \eta\cdot C \cdot \Big( \frac{1}{m^{1-p}} + 1 \Big) \cdot \bigg( \sum_{i=1}^m \big|d^{(t)}_i \big| \bigg)^{2-p} \\
			&\leq 2\eta\cdot C \cdot \bigg( \sum_{i=1}^m \big|d^{(t)}_i \big| \bigg)^{2-p}.
		\end{split}
	\end{equation}
	In the above, we used the definitions of $\epsilon^{(t)}, \sigma^{(t)}, r_i^{(t)}, d_i^{(t)}$, and $S^*$. We have proved \eqref{eq:Lp-rates}  since $\sum_{i=1}^m \big|d^{(t+1)}_i\big| \leq (1+\eta) \sum_{i\in (S^*)^c}\big|d^{(t+1)}_i\big|$ and by the definition of $C$.
\end{proof}

\section{Auxiliary Theoretical Results}
\begin{lemma}[Upper Bound of The Residual]\label{lemma:1}
	Recall $\br^* = \bA \bx^* - \by$. Run $\texttt{IRLS}_1$ Algorithm \ref{algo:IRLS} with $p=1$ and the update rule \eqref{eq:update-Kummerle} for $\sigma^{(t)}$ and $\epsilon^{(t)}$, which yields the iterates $\{\bx^{(t)}\}_{t>0}$. If $A$ satisfies $(k,\eta)$-stable RSP, then for every $t\geq 1$ the residual $\br^{(t)}:= \bA \bx^{(t)} - \by$ satisfies
	\begin{align}
		H_{\epsilon^{(t)}}\big(\br^{(t)}\big) - \norm{\br^*}{1} \leq 3\sigma^{(t)},
	\end{align}
	where $H_{\epsilon^{(t)}}\big(\br^{(t)}\big)$ is the smoothed $\ell_p$-objective defined in \eqref{eq:Heps:def} with $p=1$.
\end{lemma}
\begin{proof}
	Recall $R:=\{ i: \big|r_i^{(t)}\big|>\epsilon\}\subset \{1,\dots,m\}$ and the definition of $H_{\epsilon^{(t)}}$ \eqref{eq:Heps:def}, we have 
	\begin{align}
		H_{\epsilon^{(t)}}\big(\br^{(t)}\big) - \norm{\br^*}{1} &= \sum_{i\in R} r^{(t)}_i + \sum_{i\in R^c} \frac{1}{2}\Big( \frac{r^{(t)}_i}{\epsilon^{(t)}}  + \epsilon^{(t)} \Big)  - \norm{\br^*}{1} \\
		&\leq \norm{\br^{(t)}}{1} + |R^c| \cdot \epsilon^{(t)} - \norm{\br^*}{1} \\
		&\leq \sigma^{(t)} + \norm{\br^{(t)}}{1}  - \norm{\br^*}{1}.
	\end{align}
	Then it remains to prove $\norm{\br^{(t)}}{1}  - \norm{\br^*}{1}\leq 2 \sigma^{(t)}$. For this, define $S^{(t)}$ to be the support of the $k$ largest entries of $\br^{(t)}$ in absolute values, then $\sum_{i\in (S^{(t)})^c} \big| r^{(t)}_i  \big| =\sigma^{(t)}$, and we have that
	\begin{align}
		\sum_{i\in (S^{(t)})^c} \big| r^{(t)}_i - r^*_i \big| &\leq \sum_{i\in (S^{(t)})^c} \big| r^{(t)}_i  \big| + \sum_{i\in (S^{(t)})^c} \big|r^*_i \big| \\
		&=  \sigma^{(t)} + \norm{\br^*}{1} -  \sum_{i\in S^{(t)}} \big|r^*_i \big| \\ 
		&=2\sigma^{(t)} + \norm{\br^*}{1} - \norm{\br^{(t)}}{1} + \sum_{i\in S^{(t)}} \big|r^{(t)} \big| - \sum_{i\in S^{(t)}} \big|r^*_i \big| \\
		&\leq 2\sigma^{(t)} + \norm{\br^*}{1} - \norm{\br^{(t)}}{1} +\sum_{i\in S^{(t)}} \big| r^{(t)}_i - r^*_i \big|. \label{eq:ub-Stc}
	\end{align}
	Since $\bA$ satisfies $(k,\eta)$-stable RSP and $\br^{(t)}-\br^*$ is in the range space of $\bA$, it holds that
	\begin{align}
		&\sum_{i\in S^{(t)}} \big| r^{(t)}_i - r^*_i \big| \leq \eta \sum_{i\in (S^{(t)})^c} \big| r^{(t)}_i - r^*_i \big| \\
		\Rightarrow & \sum_{i\in S^{(t)}} \big| r^{(t)}_i - r^*_i \big| \leq \frac{\eta}{1-\eta} \Big( 2\sigma^{(t)} + \norm{\br^*}{1} - \norm{\br^{(t)}}{1} \Big). \label{eq:ub-St}
	\end{align}
	In the last inequality we used \eqref{eq:ub-Stc} and rearranged terms. Using  \eqref{eq:ub-Stc} and \eqref{eq:ub-St} we obtain
	\begin{align}
	    \norm{\br^{(t)}}{1}  - \norm{\br^*}{1} &\leq \norm{\br^{(t)} - \br^*}{1} \\ 
	    &= \sum_{i\in (S^{(t)})^c} \big| r^{(t)}_i - r^*_i \big| + \sum_{i\in S^{(t)}} \big| r^{(t)}_i - r^*_i \big| \\
	    &\leq 2\sigma^{(t)} + \norm{\br^*}{1} - \norm{\br^{(t)}}{1} + 2 \sum_{i\in S^{(t)}} \big| r^{(t)}_i - r^*_i \big|  \\
	    &\leq \frac{1+\eta}{1-\eta} \Big( 2\sigma^{(t)} + \norm{\br^*}{1} - \norm{\br^{(t)}}{1} \Big) 
	\end{align}
 	which implies 
	\begin{align}
		\norm{\br^{(t)}}{1}  - \norm{\br^*}{1} \leq \frac{1+\eta}{2} \cdot 2\sigma^{(t)} \leq 2\sigma^{(t)}.
	\end{align}
	This finishes the proof.
\end{proof}

\subsection{The Stable Range Space Property of Gaussian Matrices}\label{appendix:RSP-Gaussian}
\begin{proof}[Proof of Proposition \ref{prop:Gaussian-RSP}]
	We need the notion of \textit{Gaussian width} $w(\cdot)$ of a given set $\cK$, defined as
	\begin{align}\label{eq:Gw}
		w(\cK) := \bbE_{\bg\sim \cN(0,\bI_m)}\bigg[ \sup_{\br \in \cK} \trsp{\bg}\br \bigg].
	\end{align}
	
	Write $\bbS^{m-1}:=\{v\in\bbR^m: \trsp{v} v=1 \}$. Denote by $\text{Gr}(n,m)$ the set of $n$ dimensional subspaces of $\bbR^m$, also known as the \textit{Grassmannian} manifold or variety. With  $\eta\in(0,1]$, consider the set
	\begin{align}
		\cT_k := \bigg\{ \br\in\bbR^m: \sum_{i\in S} |r_i| > \eta \sum_{i\in S^c} |r_i|,\  \text{for some $S\subset\{1,\dots,m\}$ with of cardinality $k$ } \bigg\}.
	\end{align}
	The following lemma gives an upper bound on the Gaussian width of $\cT_k \cap \bbS^{m-1}$:
	\begin{lemma}[Remark 9.30 and Proposition 9.33 of \cite{Foucart-2013}]\label{lemma:Foucart-2013-9.33}
		If $m\geq 2k$ then we have:
		\begin{align}
			w\big(\cT_k \cap \bbS^{m-1} \big) \leq \sqrt{2k \ln (em/k)} \cdot \Big( 1.67 + \eta^{-1} \Big).
		\end{align}
	\end{lemma}
	Lemma \ref{lemma:Foucart-2013-9.33} and \eqref{eq:condition-Gaussian-RSP} imply that $w\big(\cT_k \cap \bbS^{m-1} \big) \leq \frac{m-n}{\sqrt{m-n+1}}$. Since  $\bA$ has i.i.d. $\cN(0,1)$ entries, we can think of its range space $\text{r}(\bA)$ as drawn uniformly at random from the Grassmannian $\text{Gr}(n,m)$. Invoking Lemma \ref{lemma:Gordon} with $\cV=\text{r}(\bA)$ and $\cM = \cT_k \cap \bbS^{m-1}$, we see that 
	\begin{align}
		\text{Pr}\big(\text{r}(\bA) \cap \cT_k \cap \bbS^{m-1} = \varnothing\big) \geq 1- 2.5\exp\bigg( -\frac{1}{18}\Big(\frac{m-n}{\sqrt{m-n+1}} - w(\cT_k \cap \bbS^{m-1})  \Big)^2  \bigg).
	\end{align}
	Furthermore, Lemma \ref{lemma:Foucart-2013-9.33} and condition \eqref{eq:condition-Gaussian-RSP} make sure that 
	\begin{align}
		\frac{m-n}{\sqrt{m-n+1}} - w(\cT_k \cap \bbS^{m-1}) &\geq  \frac{m-n}{\sqrt{m-n+1}} - \sqrt{2k \ln (em/k)} \cdot \Big( 1.69 + \eta^{-1} \Big) \\
		&\geq \sqrt{18\ln (2.5\delta^{-1})}.
	\end{align}
	Combining the above leads us to 
	\begin{align}
		\text{Pr}\big(\text{r}(\bA) \cap \cT_k \cap \bbS^{m-1} = \varnothing\big) \geq 1- 2.5 \exp \big( -\ln(2.5\delta^{-1}) \big) = 1-\delta.
	\end{align}
	The event $\text{r}(\bA) \cap \cT_k \cap \bbS^{m-1} = \varnothing$ implies that the stable $(k,\eta)$-stable RSP holds true. 
\end{proof}

\begin{lemma}[Gordon's Escape Through a Mesh Theorem, Corollary 3.4 of \cite{Gordon-1988}, Theorem 4.3 of \cite{Rudelson-CPAM2008}]\label{lemma:Gordon}
	Let $\cV$ be a $n$-dimensional subspace of $\bbR^m$ drawn uniformly at random from the Grassmannian $\text{Gr}(n,m)$. Let $\cM$ be a subset\footnote{The proof of Gordon \cite{Gordon-1988} assumes $\cM$ to be closed, but this assumption can be removed in view of the definition of the Gaussian width and the compactness of $\bbS^{m-1}$. The factor $3.5$ of \cite{Gordon-1988} can be improved to $2.5$ \cite{Rudelson-CPAM2008}. To correct \cite{Rudelson-CPAM2008}; their condition ``$w(S)< \sqrt{k}$'' should  be replaced by ``$w(S)< k/\sqrt{k+1}$''.
	} of $\bbS^{m-1}$. If $w(\cM) < \frac{m-n}{\sqrt{m-n+1}}$ then
	\begin{align}
		 \text{Pr}\big(\cV\cap \cM = \varnothing\big) \geq 1- 2.5\exp\bigg( -\frac{1}{18}\Big(\frac{m-n}{\sqrt{m-n+1}} - w(\cM)  \Big)^2  \bigg).
	\end{align}
\end{lemma}



\section{Experimental Setup}\label{section:experimental-setup}

\subsection{How We Run Other Methods}
In the real phase retrieval experiment (Figure \ref{fig:RPR}), all other baselines are implemented in PhasePack \cite{Chandra-SampTA2019}. We use the following (quite standard) parameters:
\begin{verbatim}
opts.tol = 1e-11; opts.initMethod = 'Truncatedspectral';
\end{verbatim}

In the experiments of linear regression without correspondences (Figures \ref{fig:SLR1} and \ref{fig:SLR2}), the implementation of PDLP \cite{Applegate-NeurIPS2021} is \hyperlink{https://github.com/google-research/FirstOrderLp.jl}{here}, and we run it with the command: 
\begin{verbatim}
julia --project=scripts scripts/solve_qp.jl \
--instance_path test/trivial_lp_model.mps --iteration_limit 500 \
--method pdhg --output_dir [my directory]
\end{verbatim}
Note that we run PDLP for a maximum of $500$ iterations, while the recommended number of maximum iterations from their GitHub repo is $5000$; this is mainly for the sake of efficiency. We use the newest version 9.5.1 of the Gurobi solver with default parameters. we employ the implementation of the FOM toolbox \cite{Beck-OMS2019} for the (proximal) subgradient descent method. In particular, we invoke the function $\texttt{prox\_subgradient}$ with the function $G$ being zero and $F$ being $\|\bA \bx - \by\|_1$. We use $0$ as initialization, and we observe similar performance when using the least-squares initialization. We set $\texttt{par.alpha}$ to be $1/\|\bA\|_2$, which corresponds to a stepsize of $1/\|\bA\|_2 / (t+1)$, where $t$ is the number of iterations. We make this choice because the default $\texttt{par.alpha}=1$ does not work well and sometimes diverges. Finally, we set the maximum number of iterations to be $10000$, as the default choice $1000$ leads to an estimate with large errors.

\subsection{Synthetic Data Generation}\label{subsection:data-generation}
In this section, we provide the Matlab codes generating synthetic data for experiments visualized in Figures \ref{fig:RR1}, \ref{fig:application}, and\ref{fig:LpRR_noiseless_noisy}, detailing the precise data generation process.
\begin{verbatim}
%% robust regression
function [y, A, x] = gen_RR(m, n, k, sigma)  
    A = randn(m, n); x = randn(n,1);
    
    y = zeros(m,1);
    
    idx = datasample(1:m, k, 'Replace', false);
    
    y(idx) = randn(k,1);
    
    s = setdiff(1:m, idx);
    y(s) =  A(s, :) * x + sigma * randn(length(s), 1);
end

%% linear regression without correspondences 
function [y, A, x] = gen_SLR(m, n, sigma, shuffle_ratio)
    A = randn(m, n); x = randn(n,1);
    
    w = sigma*randn(m, 1);
    
    y0 = A*x;
    k = int64(shuffle_ratio * m);
    partial_idx = datasample(1:m, k, 'Replace', false);
    y1 = y0(partial_idx, 1);
    y0(partial_idx, 1) = y1(randperm(k));
    
    y = y0 + w;
end

%% real phase retrieval
function [y, A, x] = gen_RPR(m, n, num_positive_sign)
    A = randn(m, n); x = randn(n,1); y = zeros(m,1);
    
    idx = datasample(1:m, num_positive_sign, 'Replace', false);
    
    y(idx) = A(idx, :) * x;
    
    s = setdiff(1:m, idx);
    y(s) = - A(s, :) * x;
    
    % make y positive
    idx = y < 0;
    y(idx) = -y(idx);
    
    A(idx,:) = -A(idx,:);
end

\end{verbatim}

\section{Implementation}\label{appendix:impl}
The below is our Matlab code that implements $\texttt{IRLS}_p$:
\begin{verbatim}
function [x_hat] = IRLSp(A, y, p, k, num_iter)
    [m, n] = size(A);
    
    q = 2 - p;  l = m - k; epsilon = inf;    

    w = ones(m,1); x_old = zeros(n,1);    
    
    for i = 1:num_iter        
        x_hat = (A'* (w.*A)) \  (A' * (w.* y));      

        abs_residual = abs(A*x_hat - y);
        
        sigma = sum(mink(abs_residual, l)) / m;        
        
        epsilon = max(min(epsilon, sigma), 1e-16);
        
        w = 1./ (max(abs_residual, epsilon)).^(q);
      
        if norm(x_old - x_hat)/ norm(x_hat) < 1e-15
            break;
        end
        x_old = x_hat;
    end
end
\end{verbatim}

\end{document}